\documentclass[a4paper,12pt,intlimits,oneside,reqno]{amsart}
\usepackage{enumerate}
\usepackage{amsfonts,amsmath, amsxtra,amssymb,latexsym, amscd,amsthm}

\usepackage{latexsym,amssymb}
\newtheorem{theorem}{Theorem}[section]

\newtheorem{lemma}[theorem]{Lemma}

\theoremstyle{corollary}
\newtheorem{corollary}[theorem]{Corollary}

\theoremstyle{definition}
\newtheorem{definition}[theorem]{Definition}

\theoremstyle{remark}

\numberwithin{equation}{section}

\newcommand{\comment}[1]{}
\allowdisplaybreaks

\begin{document}

\title [Rough Hausdorff operator]{Weighted Morrey-Herz space estimates for rough Hausdorff operator and its commutators }

\author{Nguyen Minh Chuong}

\address{Institute of mathematics, Vietnamese  Academy of Science and Technology,  Hanoi, Vietnam.}
\email{nmchuong@math.ac.vn}
\thanks{This paper is supported by Vietnam National Foundation for Science and Technology Development (NAFOSTED)}

\author{Dao Van Duong}
\address{School of Mathematics,  Mientrung University of Civil Engineering, Phu Yen, Vietnam.}
\email{daovanduong@muce.edu.vn}

\author{Nguyen Duc Duyet}
\address{Vinh Phuc vocational college, Vinh Phuc, Vietnam.}
\email{duyetnguyenduc@gmail.com}
\begin{abstract}
 In this paper, we give necessary and sufficient conditions for the boundedness of rough Hausdorff operators on Herz, Morrey and Morrey-Herz spaces with absolutely homogeneous weights. Especially, the estimates for operator norms in each case are worked out. Moreover, we also establish the boundedness of the commutators of rough Hausdorff operators on the two weighted Morrey-Herz type spaces with their symbols belonging to Lipschitz space.
\end{abstract}

\maketitle

\section{Introduction}\label{section1}
Let $\Phi(t)$ be a locally integrable function in $(0,\infty)$. The one dimensional Hausdorff operator is defined in terms of the integral form as follows
\begin{align}
{\mathcal H}_{\Phi}f(x)=\int_0^{\infty}\frac{\Phi(t)}{t}f\left(\frac{x}{t} \right)dt.
\end{align}
It is well known that the Hausdorff operator is one of important operators in harmonic analysis, and it closely related to the summability of the classical Fourier series.  It is worth pointing out that if the kernel function $\Phi$ is taken appropriately, then  the Hausdorff operator reduces to many classcial operators in analysis such as the Hardy operator, the Ces\`{a}ro operator, the Riemann-Liouville fractional integral operator and the Hardy-Littlewood average operator (see, e.g., \cite{Andersen}, \cite{Christ}, \cite{FGLY2015}, \cite{Miyachi} and references therein).
\vskip 5pt
In 2002, Brown and M\'{o}ricz \cite{BM} extended the study of  Hausdorff operator to the high dimensional space which is defined as follows
\begin{equation}\label{Hausdorff1}
H_{\Phi, A}(f)(x)=\int\limits_{\mathbb R^n}{\frac{\Phi(t)}{|t|^n}f(A(t) x)dt},\,x\in\mathbb R^n,
\end{equation}
where $\Phi$ is a locally integrable function on $\mathbb R^n$, and $A(t)$ is an $n\times n$ invertible matrix for almost everywhere $t$ in the support of $\Phi$.  It should be pointed out that if  the kernel function $\Phi$ and $A(t)$ are chosen suitably, then $H_{\Phi,A}$ reduces to the weighted Hardy-Littlewood average operator, the weighted Hardy-Ces\`{a}ro operator (see \cite{Carton-Lebrun, CH2014}). More generally, Chuong, Duong and Dung \cite{HausdoffCDD} recently have introduced a more general multilinear operator of Hausdorff type
\begin{align*}
H_{\Phi,\vec{A}}(\vec{f})(x)=\int_{\mathbb{R}^n}\frac{\Phi(t)}{|t|^n}\prod_{i=1}^m f_i(A_i(t)x)dt,x\in\mathbb{R}^n,
\end{align*}
where $\Phi:\mathbb{R}^n\to [0,\infty)$ and $A_i(t)$, for $i=1,...,m$ are $n\times n$ invertible matrices for almost everywhere $t$ in the support of $\Phi$, and $f_1,f_2,...,f_m:\mathbb{R}^n\to \mathbb{C}$ are measurable functions.
\vskip 5pt
It is interesting to see that the theory of weighted Hardy-Littlewood average operators, Hardy-Ces\`{a}ro operators and  Hausdorff operators has been significantly developed into different contexts (for more details see \cite{CH2014},  \cite{BM}, \cite{Moricz2005},  \cite{HausdoffCDD}, \cite{CHH2017},  \cite{Xiao2001} and references therein). In 2016, Chuong, Duong and Hung \cite{CDH2016} studied the boundedness of the Hardy-Ces\`{a}ro operators and their commutators on weighted Herz, Morrey and Morrey-Herz spaces with absolutely homogeneous weights. In  2012, Chen, Fan and Li \cite{CFL2012} introduced another version of Hausdorff operators, so-called the rough Hausdorff operators, as follows  
\begin{equation}\label{RoughHausdorf1}
{{\mathcal H}}_{\Phi,\Omega}(f)(x)=\int_{\mathbb R^n}\dfrac{\Phi\left(x|y|^{-1}\right)}{|y|^n}\Omega\left(y|y|^{-1}\right)f(y)dy,\;\;x\in\mathbb R^n,
\end{equation}
where $\Phi: \mathbb R^n\longrightarrow \mathbb C $ and $\Omega: S^{n-1}\longrightarrow \mathbb C $ are Lebesgue measurable functions.
Note that if $\Phi$ is a radial function, then by using the change of variable in polar coordinates, the operator ${{\mathcal H}}_{\Phi,\Omega}$ is rewritten in terms of  the following form
\begin{equation}\label{RoughHausdorf}
{{\mathcal H}}_{\Phi,\Omega}(f)(x)=\int_0^{\infty}\int_{S^{n-1}} \dfrac{\Phi(t)}{t}\Omega(y')f(t^{-1}|x|y')d\sigma(y')dt.
\end{equation}
It is useful to remark that if we choose $\Phi(t)=t^{-n}\chi_{(1,\infty)}(t)$ and $\Omega \equiv 1$, the rough Hausdorff operator ${{\mathcal H}}_{\Phi,\Omega}$ reduces to the famous Hardy operator
 \begin{equation}\label{hardyoperator}
{{\mathcal H}}(f)(x)=\frac{1}{|x|^n}\int_{|y|\leq |x|}f(y)dy.
\end{equation}
Also, if  $\Omega \equiv 1$ and $\Phi(t)=\chi_{(0, 1)}(t)$, the ${{\mathcal H}}_{\Phi,\Omega}$ reduces to the adjoint Hardy operator 
 \begin{equation}\label{hardyoperator}
{{\mathcal H^\star}}(f)(x)=\int_{|y|> |x|} \frac{f(y)}{|y|^n}dy.
\end{equation}
Moreover, Chen, Fan and Li \cite{CFL2012} revealed that the rough Hausdorff operators have better performance on the Herz type Hardy spaces $H{\dot{K}}_q^{\alpha,p}(\mathbb{R}^n)$ than their performance on the Hardy spaces $H^p(\mathbb{R}^n)$ when $0 < p < 1$. Meanwhile, the authors obtained some new results and generalized some known results for the high dimensional Hardy operator as well as the adjoint Hardy operator.
\vskip 5pt
 Let $b$ be a measurable function. Let $\mathcal{M}_b$ be the multiplication operator defined  by $\mathcal{M}_bf (x)=b(x) f (x)$ for any measurable function $f$. If $\mathcal{H}$ is a linear operator on some measurable function space, the commutator of Coifman-Rochberg-Weiss type formed by $\mathcal{M}_b$  and $\mathcal{H}$ is defined by $[\mathcal{M}_b, \mathcal{H}]f (x)=(\mathcal{M}_b\mathcal{H}-\mathcal{H}\mathcal{M}_b) f (x)$. In particular, if  $\mathcal{H}={{\mathcal H}}_{\Phi,\Omega}$, then we have the commutators of Coifman-Rochberg-Weiss type of the rough Hausdorff operator given as follows
\begin{align}\label{Commu}
{{\mathcal H}}_{\Phi,\Omega}^bf(x)&=b(x){{\mathcal H}}_{\Phi,\Omega}f(x)-{{\mathcal H}}_{\Phi,\Omega}(bf)(x)\notag\\
&=\int\limits_0^{\infty}\int\limits_{S^{n-1}}\frac{\Phi(t)}{t}\Omega(y')f(|x|t^{-1}y') \left[b(x)-b(|x|t^{-1}y') \right]d\sigma(y') dt.
\end{align}
\vskip 5pt
Inspired by above mentioned results, the goal of this paper is to extend and develop the known results in \cite{CDH2016} to rough Hausdorff operators setting. More precisely, we establish the necessary and sufficient conditions for the boundedness of rough Hausdorff operators  on weighted Herz, central Morrey, and Morrey-Herz spaces with absolutely homogeneous weights. In each case, the estimates for operator norms are worked out. Also, the sufficient conditions for the boundedness of the commutators of rough Hausdorff operators on the two weighted Morrey-Herz type spaces with their symbols belonging to Lipschitz space is given.
\vskip 5pt
Our paper is organized as follows. In Section 2, we give necessary preliminaries for Herz spaces, central Morrey spaces and Morrey-Herz spaces as well as the class of absolutely homogeneous weights. Our main theorems are given and proved in Section 3.
\section{Preliminaries}\label{section2}

Before stating our results in the next section, let us give some basic facts and notations which will be used throughout this paper. By $\|T\|_{X\to Y}$, we denote the norm of $T$ between two normed vector spaces $X,Y$. The letter $C$ denotes a positive constant which is independent of the main parameters, but may be different from line to line. For any $a\in\mathbb R^n$ and $r>0$, we shall denote by $B(a,r)$ the ball centered at $a$ with radius $r$. We also denote $S^{n-1}=\{ x\in\mathbb{R}^n: |x|=1\}$ and $|S^{n-1}|=\frac{2\pi^{\frac{n}{2}}}{\Gamma\left( \frac{n}{2}\right)}$. For any real number $p>0$, denote by $p'$ conjugate real number of $p$, i.e. $\frac{1}{p}+\frac{1}{p'}=1$.
\vskip 5pt
Next, we write $a \lesssim b$ to mean that there is a positive constant $C$ , independent of the main parameters, such that $a\leq Cb$. The symbol $f\simeq g$ means that $f$ is equivalent to $g$ (i.e.~$C^{-1}f\leq g\leq Cf$). Thoughout the paper,  the weighted function $\omega(x)$ will be denoted a nonnegative measurable function on $\mathbb{R}^n$,  and let $L^q_{\omega}(\mathbb{R}^n)$ $(0<q<\infty)$ be the space of all Lebesgue measurable functions $f$ on $\mathbb{R}^n$ such that
\begin{align*}
\|f\|_{q,\omega}=\left(\int_{\mathbb{R}^n}|f(x)|^q\omega(x)dx \right)^{\frac{1}{q}}<\infty.
\end{align*} 
The space $L^q_\text {loc}(\omega, \mathbb R^n)$ is defined as the set of all measurable functions $f$ on $\mathbb R^n$ satisfying $\int_{K}|f(x)|^q\omega(x)dx<\infty$ for any compact subset $K$ of $\mathbb R^n$. The space $L^q_\text {loc}(\omega, \mathbb R^n\setminus\{0\})$ is also defined in a similar way to the space  $L^q_\text {loc}(\omega, \mathbb R^n)$.
\vskip 5pt
In the following definitions $\chi_k=\chi_{C_k}$, $C_k=B_k\setminus B_{k-1}$ and $B_k = \big\{x\in \mathbb R^n: |x| \leq 2^k\big\}$, for all $k\in\mathbb Z.$ Now, we are in a position to give some definitions of the Lipschitz , Herz, Morrey and Morrey-Herz spaces. For further information on these spaces as well as their deep applications in analysis, the interested readers may refer to the work \cite{ALP2000} and to the monograph \cite{LYH2008}.

\begin{definition}
Let $0<\beta\leq 1$. The Lipschitz space $Lip^\beta(\mathbb{R}^n)$ is defined as the set of all functions $f:\mathbb{R}^n\to \mathbb{C}$ such that
\begin{align*}
\|f\|_{Lip^\beta(\mathbb{R}^n)}:=\sup_{x,y\in \mathbb{R}^n,\,x\ne y}\frac{|f(x)-f(y)|}{|x-y|^\beta}<\infty.
\end{align*}
\end{definition}

\begin{definition}
Let $\lambda \in \mathbb R$ and $1 \leq p < \infty$. The weighted central Morrey space ${\mathop B\limits^.}_{\omega}^{p,\lambda}(\mathbb R^n)$
is defined as the set of all locally $p$-integrable functions $f$ satisfying
\[
{\big\| f \big\|_{{\mathop B\limits^.}_{\omega}^{p,\lambda }({\mathbb R^n})}} = \mathop {\sup }\limits_{R\, > 0} {\Big( {\frac{1}{{{\omega}\big(B(0,R)\big)^{1+\lambda p}}}\int\limits_{B(0,R)} {{{\left| {f(x)} \right|}^p}\omega(x)dx} } \Big)^{1/p}} < \infty .
\]
\end{definition}

\begin{definition}
Let $\alpha\in\mathbb R$, $0<q<\infty$, and $0<p<\infty$. The weighted homogeneous Herz-type space $\dot{K}^{\alpha, p}_q(\omega)$ is defined by
\[
\dot{K}^{\alpha, p}_q(\omega)=\big\{f\in L^q_\text {loc}(\mathbb R^n\setminus\{0\},\omega):\|f\|_{\dot{K}^{\alpha,p}_q(\omega)} <\infty \big \},
\]
where $\|f\|_{\dot{K}^{\alpha, p}_q(\omega)}=\Big(\sum\limits_{k=-\infty}^\infty 2^{k\alpha p} \|f\chi_k\|^p_{q,\omega}\Big)^{\frac{1}{p}}.$
\end{definition}

\begin{definition}
Let $\alpha \in \mathbb R$, $0 < p < \infty, 0 < q <\infty, \lambda \geq 0$ and $\omega$ be non-negative weighted function. The homogeneous weighted Morrey-Herz-type space $M\dot{K}_{p,q}^{\alpha,\lambda}(\omega)$ is defined by
\[
M\dot{K}^{\alpha,\lambda}_{p,q}(\omega)=\big\{f\in L^q_\text {loc}(\mathbb R^n\setminus\{0\},\omega):\|f\|_{M\dot{K}^{\alpha,\lambda}_{p,q}(\omega)} <\infty \big \},
\]
where $\|f\|_{M\dot{K}^{\alpha,\lambda}_{p,q}(\omega)}=\sup\limits_{k_0\in\mathbb Z}2^{-k_0\lambda}\Big(\sum\limits_{k=-\infty}^{k_0} 2^{k\alpha p} \|f\chi_k\|^p_{q,\omega}\Big)^{\frac{1}{p}}.$
\end{definition}
Note that $\dot{K}^{0,p}_{p}(\mathbb R^n)= L^p(\mathbb R^n)$ for $0<p<\infty$, and  $\dot{K}^{\alpha/p,p}_{p}(\mathbb R^n)= L^p(|x|^\alpha dx)$ for all $0<p<\infty$ and $\alpha\in\mathbb R$. Since $M\dot{K}^{\alpha,0}_{p,q}(\mathbb R^n)$ = $\dot{K}^{\alpha, p}_{q}(\mathbb R^n)$, it follows that the Herz spaces are the special cases of Morrey-Herz spaces. Therefore, it is said that the Herz spaces and Morrey-Herz spaces are natural generalizations of the Lebesgue spaces associated with power weights.
\vskip 5pt
Next, let us give some definitions of the two weighted Herz, Morrey, and Morrey-Herz spaces.
\begin{definition}
Let $0<p<\infty$ and $\lambda>0$. Suppose $\omega_1, \omega_2$ are  two weighted functions. Then,  
the two weighted Morrey space is defined by
\begin{align*}
\dot{B}^{p,\lambda}(\omega_1, \omega_2)=\{f\in L^p_{\rm loc}(\omega_1): \|f\|_{\dot{B}^{p,\lambda}(\omega_1, \omega_2)}<\infty \},
\end{align*}
where
\begin{align*}
\|f\|_{\dot{B}^{p,\lambda}(\omega_1, \omega_2)}=\sup\limits_{R>0}\left(\frac{1}{\omega_2(B(0,R))^{\lambda}}\int_{B(0,R)}|f(x)|^p\omega_1(x)dx \right)^{\frac{1}{p}}.
\end{align*}
\end{definition}

\begin{definition}
Let $0<p<\infty$, $0<q<\infty$, and $\alpha\in\mathbb{R}$. Let $\omega_1$ and $\omega_2$ be nonnegative weighted functions. The homogeneous two weighted Herz space $\dot{K}_q^{\alpha,p}(\omega_1,\omega_2)$ is defined to be the set of all $f\in L_{\rm loc}^q(\mathbb{R}^n\backslash\{0\};\omega_2)$ such that
\begin{align*}
\|f\|_{\dot{K}_q^{\alpha,p}(\omega_1,\omega_2)}=\left(\sum\limits_{k\in\mathbb{Z}}\omega_1(B_k)^{\frac{\alpha}{n}p}\|f\chi_k\|^p_{L^q(\mathbb{R}^n;\omega_2)} \right)^{\frac{1}{p}}<\infty.
\end{align*}
\end{definition}

\begin{definition}
Let $\alpha\in\mathbb{R}, 0<p< \infty, 0<q< \infty, \lambda\ge 0$ and $\omega_1, \omega_2$ be weighted functions. The two weighted Morrey-Herz space $M\dot{K}_{p,q}^{\alpha,\lambda}(\omega_1,\omega_2)$ is defined as the space of all functions $f\in L_{\rm loc}^q(\mathbb{R}^n\backslash\{0\};\omega_2)$ such that $\|f\|_{M\dot{K}_{p,q}^{\alpha,\lambda}(\omega_1,\omega_2)}<\infty$, where
\begin{align*}
\|f\|_{M\dot{K}_{p,q}^{\alpha,\lambda}(\omega_1,\omega_2)}=\sup\limits_{k_0\in\mathbb{Z}}\left(\omega_1(B_{k_0})^{-\frac{\lambda}{n}}\left(\sum\limits_{k=-\infty}^{k_0}\omega_1(B_k)^{\frac{\alpha}{n}p}\|f\chi_k\|^p_{q,\omega_2)} \right)^{\frac{1}{p}}\right).
\end{align*}
\end{definition}

It is obvious that for $\lambda=0$, we have $M\dot{K}_{p,q}^{\alpha,0}(\omega_1,\omega_2)=\dot{K}_{q}^{\alpha, p}(\omega_1,\omega_2)$. Also, note that if we take $\omega_1(x)=|B_0|^{-1}$,  then $M\dot{K}_{p,q}^{\alpha,\lambda}(\omega_1,\omega_2)$ reduces to  the usual one weighted Morrey-Herz space $MK_{p,q}^{\alpha,\lambda}(\omega)$. For further the applications of these spaces in analysis, the readers can refer to the monograph \cite{LYH2008}.
\begin{definition}
Let $\gamma$ be a real number. Let $\mathcal{W}_{\gamma}$ be the set of all Lebesgue measurable functions $\omega$ on $\mathbb{R}^n$ such that $\omega(x)>0$ for almost every where $x\in\mathbb{R}^n, 0<\int_{S^{n-1}}\omega(x)\sigma(x)<\infty$, and $\omega$ is absolutely homogeneous of degree $\gamma$, that is, $\omega(tx)=|t|^{\gamma}\omega(x)$ for all $t\in\mathbb{R}\backslash \{0\},\; x\in\mathbb{R}^n$.
\end{definition}
Let us denote $\mathcal{W}=\bigcup\limits_{\gamma}\mathcal{W}_{\gamma}$. It is easy to see that   $\mathcal{W}$ contains strictly the class of power weights of the form $|x|^{\gamma}$.  For further discussions, the readers can refer to \cite{CH2014} and \cite{CDH2016}.  Throughout the whole paper, we will denote by $\omega $ a weight in $\mathcal{W}_\gamma$. Next, we recall the following result related to the class of weight functions $\mathcal{W}_\gamma $, which is used in the sequel.
\begin{lemma}[\cite{CDH2016}]\label{lemma}
Let $\omega\in \mathcal{W}_\gamma$ for $\gamma>-n$.  Then, there exists a constant $C=C(\omega,n)>0$ such that
$$\omega(B_m)=C|B_m|^{\frac{\gamma+n}{n}}\;\; \text {and }\;\;
\omega(C_m)=(1-2^{-\gamma-n})\omega(B_m),
$$
for any $m\in\mathbb Z$.
\end{lemma}

\section{Main results and their proofs}
Before stating our main results, we introduce some notations which will be used throughout this section. Assume that $\Phi: \mathbb R^n\longrightarrow \mathbb C $  is a radial measurable function, that is, $\Phi(x)=\Phi(|x|)$ for all $x\in\mathbb R^n$, and $\Omega: S^{n-1}\longrightarrow \mathbb C $ is a measurable function such that $\Omega(x)\not =0$ for almost everywhere $x$ in $S^{n-1}$. Let us recall that the rough Hausdorff operator is defined by
\begin{align}
{{\mathcal H}}_{\Phi,\Omega}(f)(x)=\int_{\mathbb R^n}\dfrac{\Phi\left(x|y|^{-1}\right)}{|y|^n}\Omega\left(y|y|^{-1}\right)f(y)dy,\;\;x\in\mathbb R^n.
\end{align}
Using polar coordinates and changing variables, it is easy to see that
\begin{align}\label{eq1}
\mathcal{H}_{\Phi,\Omega}f(x)&=\int\limits_0^{+\infty}\left[\int\limits_{S^{n-1}}\frac{\Phi(t)}{t}\Omega(y')f(|x|t^{-1}y')d\sigma(y')\right]dt.
\end{align}
  For $b\in Lip^{\beta}(0<\beta\le 1$),  the commutator of Coifman-Rochberg-Weiss type of rough Hausdorff operator with the Lipschitz functions is defined as follows
\begin{align}\label{Commu}
\mathcal{H}_{\Phi,\Omega}^bf(x)&=b(x)\mathcal{H}_{\Phi,\Omega}f(x)-\mathcal{H}_{\Phi,\Omega}(bf)(x)\notag\\
&=\int\limits_0^{+\infty}\int\limits_{S^{n-1}}\frac{\Phi(t)}{t}\Omega(y')f(|x|t^{-1}y') \left[b(x)-b(|x|t^{-1}y') \right]d\sigma(y') dt,
\end{align}
where $f:\mathbb{R}^n\to \mathbb{C}$ are measurable functions.

Now, we are in a position to give the first our main results concerning the boundedness of the rough Hausdorff operator on the weighted Morrey spaces.

\begin{theorem}\label{Morrey1}
Let  $\gamma>-n, 1\le p<\infty,1+\lambda p>0, \lambda\in\mathbb{R}$ and $\Omega\in L^{p'}(S^{n-1})$.
{\rm (i)}
If  $\omega(x')\geq c>0$ for all $x' \in S^{n-1}$, and
\begin{equation}
\mathcal{C}_1=\int\limits_0^{\infty} \frac{|\Phi(t)|}{t^{1+(n+\gamma)\lambda}}dt<\infty \notag,
\end{equation}
we have ${\mathcal H}_{\Phi,\Omega}$  is a bounded operator on $\dot{B}_{\omega}^{p,\lambda}(\mathbb{R}^n)$. Moreover,
\begin{align*}
\|{\mathcal H}_{\Phi,\Omega} \|_{\dot{B}_{\omega}^{p,\lambda}(\mathbb{R}^n)\to \dot{B}_{\omega}^{p,\lambda}(\mathbb{R}^n)}\lesssim \mathcal{C}_1\|\Omega\|_{L^{p'}(S^{n-1})}.
\end{align*}
{\rm (ii)}
Conversely, suppose $\Omega\in L^{p'}(S^{n-1}, \omega(x')d\sigma(x'))$ and $\Phi$ is a real function with a constant sign in $\mathbb{R}^n$. Then, if ${\mathcal H}_{\Phi,\Omega}$ is bounded on  $\dot{B}_{\omega}^{p,\lambda}(\mathbb{R}^n)$, we have $\mathcal{C}_1<\infty$. Furthermore,
\begin{align*}
\|{\mathcal H}_{\Phi,\Omega} \|_{\dot{B}_{\omega}^{p,\lambda}(\mathbb{R}^n)\to \dot{B}_{\omega}^{p,\lambda}(\mathbb{R}^n)}\geq \mathcal{C}_1.\frac{\|\Omega\|_{L^{p'}(S^{n-1})}^{p'}}
{\|\Omega\|_{L^{p'}(S^{n-1}, \omega(x')d\sigma(x'))}^{\frac{p'}{p}}}.
\end{align*}
\end{theorem}

\begin{proof}
(i) From \eqref{RoughHausdorf} and by the Minkowski inequality, we have
\begin{align*}
&\|{{\mathcal H}}_{\Phi,\Omega}f \|_{\dot{B}_{\omega}^{p,\lambda}(\mathbb{R}^n)}\\
&=\sup\limits_{R>0}\left(\frac{1}{\omega(B(0,R))^{1+\lambda p}}\int\limits_{B(0,R)}\left|\int\limits_0^{\infty}\left(\int\limits_{S^{n-1}}\frac{\Phi(t)}{t}\Omega(y')f(|x|t^{-1}y')d\sigma(y')\right)dt\right|^p\omega(x)dx \right)^{\frac{1}{p}}\\
&\le \sup\limits_{R>0}\int\limits_{0}^{\infty}\left(\int\limits_{B(0,R)}\frac{1}{\omega(B(0,R))^{1+\lambda p}}\frac{|\Phi(t)|^p}{t^p}\left|\int\limits_{S^{n-1}}\Omega(y')f(|x|t^{-1}y')d\sigma(y')\right|^p\omega(x)dx \right)^{\frac{1}{p}}dt.
\end{align*}
Using change of variable $u=xt^{-1}$, it is easy to show that
\begin{align*}
&\|{{\mathcal H}}_{\Phi,\Omega}f \|_{\dot{B}_{\omega}^{p,\lambda}(\mathbb{R}^n)}\\
&\le \sup\limits_{R>0}\int\limits_{0}^{\infty}\frac{|\Phi(t)|}{t^{1-\frac{\gamma}{p}-\frac{n}{p}}}\left(\int\limits_{B(0,t^{-1}R)}\frac{1}{\omega(B(0,R))^{1+\lambda p}}\left|\int\limits_{S^{n-1}}\Omega(y')f(|u|y')d\sigma(y')\right|^p\omega(u)du \right)^{\frac{1}{p}}dt.
\end{align*}
Note that, by the H\"{o}lder inequality, we have
\begin{align}\label{eq2}
\int\limits_{S^{n-1}}\Omega(y')f(|u|y') d\sigma(y')&\le \left(\int\limits_{S^{n-1}}|f(|u|y')|^pd\sigma(y') \right)^{\frac{1}{p}} \left(\int\limits_{S^{n-1}}|\Omega(y')|^{p'}d\sigma(y') \right)^{\frac{1}{p'}}\notag\\
&=\left(\int\limits_{S^{n-1}}|f(|u|y')|^pd\sigma(y') \right)^{\frac{1}{p}} \|\Omega \|_{L^{p'}(S^{n-1})}.
\end{align}
Thus, we obtain
\begin{align}\label{eq3}
&\|{{\mathcal H}}_{\Phi,\Omega}f \|_{\dot{B}_{\omega}^{p,\lambda}(\mathbb{R}^n)}\le \|\Omega \|_{L^{p'}(S^{n-1})}\sup\limits_{R>0}\int\limits_{0}^{\infty}\frac{|\Phi(t)|}{t^{1-\frac{\gamma}{p}-\frac{n}{p}}}\Psi(t, R) dt,
\end{align}
where $\Psi(t, R):=\left(\displaystyle\int\limits_{B(0,t^{-1}R)}\dfrac{1}{\omega(B(0,R))^{1+\lambda p}}\left(\int\limits_{S^{n-1}}|f(|u|y')|^pd\sigma(y') \right)\omega(u)du \right)^{\frac{1}{p}}$.
Now, by putting $u=rx'$ and using the condition $\omega(x')\geq c>0$ for all $x'\in S^{n-1}$,  we have
\begin{align}\label{eq4}
\Psi(t)&=\left(\frac{1}{\omega(B(0,R))^{1+\lambda p}}\int\limits_0^{t^{-1}R}\int\limits_{S^{n-1}}\left(\int\limits_{S^{n-1}}|f(|rx'|y')|^pd\sigma(y') \right)\omega(rx')d\sigma(x')r^{n-1}dr \right)^{\frac{1}{p}}\notag\\
&=\omega(S^{n-1})^{\frac{1}{p}}\left(\frac{1}{\omega(B(0,R))^{1+\lambda p}}\int\limits_{0}^{t^{-1}R}r^{\gamma+n-1} \left(\int\limits_{S^{n-1}}|f(ry')|^pd\sigma(y') \right)dr \right)^{\frac{1}{p}}\notag\\
&\lesssim \omega(S^{n-1})^{\frac{1}{p}}\left(\frac{1}{\omega(B(0,R))^{1+\lambda p}}\int\limits_{0}^{t^{-1}R}r^{\gamma+n-1} \left(\int\limits_{S^{n-1}}|f(ry')|^p \omega(y')d\sigma(y') \right)dr \right)^{\frac{1}{p}}\notag\\
&\lesssim \omega(S^{n-1})^{\frac{1}{p}}\left(\frac{1}{\omega(B(0,R))^{1+\lambda p}}\int\limits_{B(0,t^{-1}R)} |f(x)|^p\omega(x)dx  \right)^{\frac{1}{p}}.
\end{align}
We have
\begin{align*}
\omega(B(0,{t^{-1}R}))&=\int\limits_{B(0,{t^{-1}R})}\omega(z)dz\\
&=\int\limits_{B(0,R)}t^{-(\gamma+ n)}\omega(y) dy=t^{-(\gamma+ n)}\omega(B(0,R)),
\end{align*}
so
\begin{align}\label{eq5}
\frac{1}{\omega(B(0,R))^{1+\lambda p}}=\frac{1}{t^{(\gamma+ n)(1+\lambda p)}\omega(B(0,{t^{-1}R}))^{1+\lambda p}}.
\end{align}
Hence, from \eqref{eq3},  \eqref{eq4} and \eqref{eq5}, we obtain
\begin{align*}
\|{{\mathcal H}}_{\Phi,\Omega}f \|_{\dot{B}_{\omega}^{p,\lambda}(\mathbb{R}^n)} &\lesssim \|\Omega \|_{L^{p'}(S^{n-1})}\|f\|_{\dot{B}^{p,\lambda}_{\omega}(\mathbb{R}^n)}\int\limits_{0}^{\infty}\frac{|\Phi(t)|}{t^{1+(n+\gamma)\lambda}}dt\\
&\lesssim \mathcal{C}_1 \|\Omega \|_{L^{p'}(S^{n-1})}\|f\|_{\dot{B}^{p,\lambda}_{\omega}(\mathbb{R}^n)},
\end{align*}
which completes the proof of the part (i).

(ii) Conversely, suppose ${\mathcal H}_{\Phi,\Omega}$ is bounded on the  space $\dot{B}_{\omega}^{p,\lambda}(\mathbb{R}^n)$. As it is known, a standard approach to proving for the part (ii) of the theorem is to take  a appropriately radial function.
Here, let us also choose the function as follows
\begin{align*}
f(x)=|x|^{(n+\gamma)\lambda} |\Omega(x')|^{p'-2}\,\overline{\Omega}(x'), \text { for } x'=\frac{x}{|x|}.
\end{align*}
We then have
\begin{align*}
\|f\|_{\dot{B}^{p,\lambda}_{\omega}(\mathbb{R}^n)}&=\sup\limits_{R>0}\left(\frac{1}{\omega(B(0,R))^{1+\lambda p}}\int\limits_{B(0,R)}|x|^{(n+\gamma)\lambda p}|\Omega(x)'|^{(p'-2)p}|\Omega(x')|^p\omega(x)dx \right)^{\frac{1}{p}}\\
&=\sup\limits_{R>0}\left(\frac{1}{\omega(B(0,R))^{1+\lambda p}}\int\limits_{B(0,R)}|x|^{(n+\gamma)\lambda p}|\Omega(x)'|^{p'}\omega(x)dx \right)^{\frac{1}{p}}.\\
\end{align*}
Since $\gamma>-n$, a simple computation shows that
\begin{align*}
\omega(B(0,R))=\frac{R^{n+\gamma}}{n+\gamma}\omega(S^{n-1}),
\end{align*}
and we get
\begin{align*}
\int\limits_{B(0,R)}|x|^{(n+\gamma)\lambda p}|\Omega(x')|^{p'}\omega(x)dx&=\int\limits_{0}^R\int\limits_{S^{n-1}}|rx'|^{(n+\gamma)\lambda p}|\Omega(x')|^{p'}\omega(rx')r^{n-1}d\sigma(x')dr\\
&=\left(\int\limits_{0}^Rr^{(n+\gamma)(1+\lambda p)-1}dr\right)\left(\int\limits_{S^{n-1}}|\Omega(x')|^{p'}\omega(x')d\sigma(x')\right)\\
&=\frac{R^{(n+\gamma)(1+\lambda p)}}{(n+\gamma)(1+\lambda p)}\|\Omega\|^{p'}_{L^{p'}(S^{n-1}, \omega(x')d\sigma(x'))}. 
\end{align*}
Consequently,  
\begin{align*}
\|f\|_{\dot{B}^{p,\lambda}_{\omega}(\mathbb{R}^n)}=\left(\frac{n+\gamma}{\omega(S^{n-1})} \right)^{\lambda}\frac{1}{(1+\lambda p)^{\frac{1}{p}}}\frac{1}{\omega(S^{n-1})^{\frac{1}{p}}}\|\Omega\|_{L^{p'}(S^{n-1}, \omega(x')d\sigma(x'))}^{\frac{p'}{p}}<\infty.
\end{align*}
On the other hand, by choosing $f$ as above, we get
\begin{align*}
{{\mathcal H}}_{\Phi,\Omega}f(x)
&=\int\limits_0^{\infty}\left(\int\limits_{S^{n-1}}\frac{\Phi(t)}{t}\Omega(y')f(|x|t^{-1}y') d\sigma(y')\right)dt\\
&=\int\limits_0^{\infty}\left(\int\limits_{S^{n-1}}\frac{\Phi(t)}{t}\Omega(y')\left| |x|t^{-1}y'\right|^{(n+\gamma)\lambda}|\Omega(y')|^{p'-2}\overline{\Omega}(y') d\sigma(y')\right)dt\\
&=|x|^{(n+\gamma)\lambda}.\int\limits_0^{\infty}\frac{\Phi(t)}{t}t^{-(n+\gamma)\lambda}\left(\int\limits_{S^{n-1}}|\Omega(y')|^{p'} d\sigma(y')\right)dt\\
&=\|\Omega\|_{L^{p'}(S^{n-1})}^{p'}.|x|^{(n+\gamma)\lambda}.\int\limits_0^{\infty}\frac{\Phi(t)}{t^{1+(n+\gamma)\lambda}}dt.
\end{align*}
Hence, it follows that
\begin{align*}
\|{{\mathcal H}}_{\Phi,\Omega}f \|_{\dot{B}_{\omega}^{p,\lambda}(\mathbb R^n)}\simeq \|\Omega\|_{L^{p'}(S^{n-1})}^{p'}.\||x|^{(n+\gamma)\lambda}\|_{\dot{B}_{\omega}^{p,\lambda}(\mathbb R^n)}.
\int\limits_0^{\infty}\frac{|\Phi(t)|}{t^{1+(n+\gamma)\lambda}}dt.
\end{align*}
Because $\|f\|_{\dot{B}^{p,\lambda}_{\omega}(\mathbb{R}^n)}\simeq \||x|^{(n+\gamma)\lambda}\|_{\dot{B}_{\omega}^{p,\lambda}(\mathbb R^n)}.\|\Omega\|_{L^{p'}(S^{n-1}, \omega(x')d\sigma(x'))}^{\frac{p'}{p}}$, 
we have
\begin{align*}
\|{{\mathcal H}}_{\Phi,\Omega} \|_{\dot{B}_{\omega}^{p,\lambda}(\mathbb R^n)\to \dot{B}_{\omega}^{p,\lambda}(\mathbb R^n)}&\ge \frac{\|{{\mathcal H}}_{\Phi,\Omega}f \|_{\dot{B}_{\omega}^{p,\lambda}(\mathbb R^n)}}{\|f\|_{\dot{B}_{\omega}^{p,\lambda}(\mathbb R^n)}}\\
&\gtrsim \frac{\|\Omega\|_{L^{p'}(S^{n-1})}^{p'}}
{\|\Omega\|_{L^{p'}(S^{n-1}, \omega(x')d\sigma(x'))}^{\frac{p'}{p}}}.
 \int\limits_0^{\infty}\frac{|\Phi(t)|}{t^{1+(n+\gamma)\lambda}}dt,
\end{align*}
which finishes the proof of the theorem.
\end{proof}

Remark that if $\omega(x)=|x|^\gamma$, we have $\omega(x')=1$ for all $x'\in S^{n-1}$. We then get
$$
\frac{\|\Omega\|_{L^{p'}(S^{n-1})}^{p'}}
{\|\Omega\|_{L^{p'}(S^{n-1}, \omega(x')d\sigma(x'))}^{\frac{p'}{p}}}
=\|\Omega\|_{L^{p'}(S^{n-1})}.
$$
Therefore, from Theorem \ref{Morrey1}, we have immediately the following corollary.
\begin{corollary}
Let  $\gamma>-n, 1\le p<\infty,1+\lambda p>0, \lambda\in\mathbb{R}$. Suppose  $\Omega\in L^{p'}(S^{n-1})$, $\omega(x)=|x|^\gamma$ for $\gamma>-n$, and $\Phi$ is a nonnegative radial function. Then, ${\mathcal H}_{\Phi,\Omega}$  is a bounded operator on $\dot{B}_{\omega}^{p,\lambda}(\mathbb{R}^n)$ if and only if 
$$\mathcal{C}_{1.1}=\int\limits_0^{\infty} \frac{\Phi(t)}{t^{1+(n+\gamma)\lambda}}dt<\infty.$$
 Moreover,
\begin{align*}
\|{{\mathcal H}}_{\Phi,\Omega} \|_{\dot{B}_{\omega}^{p,\lambda}(\mathbb R^n)\to \dot{B}_{\omega}^{p,\lambda}(\mathbb R^n)}\simeq\mathcal{C}_{1.1}.\|\Omega\|_{L^{p'}(S^{n-1})}.
\end{align*}
\end{corollary}
Next, we also give the boundedness and bound of the rough Hausdorff operator on the weighted Herz spaces.
\begin{theorem}\label{Herz}
Let $1\leq p, q <\infty$ and $\Omega\in L^{q'}(S^{n-1})$. \\
{\rm (i)} If $\omega(x')\ge c>0$ for all $x'\in S^{n-1}$ and 
\begin{align*}
\mathcal{C}_2= \int_{0}^\infty|\Phi(t^{-1})|\;t^{1-2n-\frac{\gamma}{q}-\frac{n}{q}}dt <  \infty,
\end{align*}
we have ${\mathcal H}_{\Phi,\Omega}$ is a bounded operator on $\dot{K}^{\alpha,p}_{q}(\omega)$. Moreover,
\begin{align*}
\|{{\mathcal H}}_{\Phi,\Omega}\|_{\dot{K}^{\alpha,p}_{q}(\omega)\to \dot{K}^{\alpha,p}_{q}(\omega)}\lesssim \mathcal{C}_2\|\Omega\|_{L^{q'}(S^{n-1})}.
\end{align*}
{\rm (ii)} Conversely, suppose $\Omega\in L^{q'}(S^{n-1}, \omega(x')d\sigma(x'))$ and $\Phi$ is a real function with a constant sign in $\mathbb{R}^n$. Then, if ${\mathcal H}_{\Phi,\Omega}$ is bounded on the space $\dot{K}^{\alpha,p}_{q}(\omega)$, we have $\mathcal{C}_2<\infty$. Furthermore, 
\[
\|{{\mathcal H}}_{\Phi,\Omega}\|_{\dot{K}^{\alpha,p}_{q}(\omega)\to \dot{K}^{\alpha,p}_{q}(\omega)}\ge \mathcal C_2.\frac{\|\Omega \|^{q'}_{L^{q'}(S^{n-1})}}{\|\Omega \|^{\frac{q'}{q}}_{L^{q'}(S^{n-1},\omega(x') d\sigma(x'))}}.
\]
\end{theorem}
\begin{proof}
(i) For every $k\in\mathbb{Z}$, by changing of variable $u=t^{-1}$, we have
\begin{align*}
\|{{\mathcal H}}_{\Phi,\Omega}f\chi_k \|_{q,\omega}&=\left(\int\limits_{\mathbb{R}^n}\vert\mathcal{H}_{\Phi,\Omega}f(x)\chi_k(x) \vert^q\omega(x)dx \right)^{\frac{1}{q}}\\
&=\left(\int\limits_{C_k}\left|\int\limits_0^{\infty}\left(\int\limits_{S^{n-1}}\Phi(u^{-1})u^{1-2n}\Omega(y')f(|x|uy')d\sigma(y')\right)du \right|^q\omega(x)dx \right)^{\frac{1}{q}}.
\end{align*}
By Minkowski's inequality and changing of variable $v=ux$, we obtain
\begin{align}\label{eq6}
\|{{\mathcal H}}_{\Phi,\Omega}f\chi_k \|_{q,\omega}&\le \int\limits_0^{\infty}|\Phi(u^{-1})|u^{1-2n-\frac{\gamma}{q}-\frac{n}{q}}\left(\int\limits_{uC_k}\left|\int\limits_{S^{n-1}}\Omega(y')f(|v|y')d\sigma(y') \right|^q\omega(v)dv \right)^{\frac{1}{q}}du.
\end{align}
On the other hand, by the H\"{o}lder inequality, we have the following estimate
\begin{align}\label{eq6a1}
\int\limits_{S^{n-1}}\Omega(y')f(|u|y') d\sigma(y')&\le \left(\int\limits_{S^{n-1}}|f(|u|y')|^pd\sigma(y') \right)^{\frac{1}{q}} \left(\int\limits_{S^{n-1}}|\Omega(y')|^{q'}d\sigma(y') \right)^{\frac{1}{q'}}\notag\\
&=\left(\int\limits_{S^{n-1}}|f(|u|y')|^pd\sigma(y') \right)^{\frac{1}{q}} \|\Omega \|_{L^{q'}(S^{n-1})}.
\end{align}
Therefore, by combining \eqref{eq6} and \eqref{eq6a1}, one has
\begin{align*}
&\|{{\mathcal H}}_{\Phi,\Omega}f\chi_k \|_{q,\omega}\notag\\
&\le\int\limits_0^{\infty}|\Phi(u^{-1})|u^{1-2n-\frac{\gamma}{q}-\frac{n}{q}}\left(\int\limits_{uC_k}\left|\left(\int\limits_{S^{n-1}}|f(|v|y')|^qd\sigma(y') \right)^{\frac{1}{q}} \|\Omega \|_{L^{q'}(S^{n-1})} \right|^q\omega(v)dv \right)^{\frac{1}{q}}du\\
&= \|\Omega \|_{L^{q'}(S^{n-1})} \int\limits_0^{\infty}|\Phi(u^{-1})|u^{1-2n-\frac{\gamma}{q}-\frac{n}{q}}\mathcal{J}(u)^{\frac{1}{q}}du,
\end{align*}
where $\mathcal{J}(u):=\int\limits_{uC_k}\left(\int\limits_{S^{n-1}}|f(|v|y')|^qd\sigma(y') \right)\omega(v)dv $. By changing of variable $v=rx'$ and using $\omega(x')\geq c>0$ for all $x'\in S^{n-1}$, we get
\begin{align*}
\mathcal{J}(u)&= \int\limits_{uC_k}\int\limits_{S^{n-1}} \left(\int\limits_{S^{n-1}}|f(|rx'|y')|^qd\sigma(y') \right)d\sigma(x') \omega(rx') r^{n-1}dr\\
&=\omega(S^{n-1})\int\limits_{uC_k}r^{n-1+\gamma} \left(\int\limits_{S^{n-1}}|f(ry')|^qd\sigma(y') \right)dr\\
&\lesssim\int\limits_{uC_k}r^{n-1+\gamma} \left(\int\limits_{S^{n-1}}|f(ry')|^q\omega(y') d\sigma(y') \right)dr =\|f\chi_{uC_k} \|_{q,\omega}^q.
\end{align*}
Thus, we obtain
\begin{align*}
\|{{\mathcal H}}_{\Phi,\Omega}f\chi_k \|_{q,\omega}\lesssim \|\Omega \|_{L^{q'}(S^{n-1})} \int\limits_0^{\infty}|\Phi(u^{-1})|u^{1-2n-\frac{\gamma}{q}-\frac{n}{q}}\|f\chi_{uC_k} \|_{q,\omega} du.
\end{align*}
Noting that for each $u\in (0, \infty)$, one can find an integer number $\ell=\ell(u)$ such that $2^{\ell-1}<u\le 2^{\ell}$. This implies that $uC_k$ is a subset of $C_{k+\ell-1}\cup C_{k+\ell}$. Thus, we obtain 
\begin{align*}
\|f\chi_{uC_k}\|_{q,\omega}\le \|f\chi_{k+\ell-1} \|_{q,\omega}+\|f\chi_{k+\ell} \|_{q,\omega}.
\end{align*}
So, one has
\begin{align}\label{eq7}
\|{{\mathcal H}}_{\Phi,\Omega}f\chi_k \|_{q,\omega}\lesssim \|\Omega \|_{L^{q'}(S^{n-1})} \int\limits_0^{\infty}|\Phi(u^{-1})|\;u^{1-2n-\frac{\gamma}{q}-\frac{n}{q}}\left( \|f\chi_{k+\ell-1} \|_{q,\omega}+\|f\chi_{k+\ell} \|_{q,\omega} \right) dt.
\end{align}
On the other hand, by $1\le p<\infty$, we have
\begin{align*}
&\|{{\mathcal H}}_{\Phi,\Omega}f \|_{\dot{K}^{\alpha,p}_{q}(\omega)}=\left(\sum\limits_{k=-\infty}^{\infty}2^{k\alpha p}\|\mathcal{H}_{\Phi,\Omega}f(x)\chi_k(x) \|_{q,\omega}^p \right)^{\frac{1}{p}}\\
&\lesssim \|\Omega \|_{L^{q'}(S^{n-1})} \int\limits_0^{\infty}|\Phi(u^{-1})|\;u^{1-2n-\frac{\gamma}{q}-\frac{n}{q}}\left(\sum\limits_{k=-\infty}^{\infty}2^{k\alpha p} \left(\|f\chi_{k+\ell-1} \|_{q,\omega}+\|f\chi_{k+\ell} \|_{q,\omega}   \right)^p \right)^{\frac{1}{p}}du.
\end{align*}
Since $2^{\ell-1}<u\le 2^{\ell}$, it follows that
\begin{align*}
&\left(\sum\limits_{k=-\infty}^{\infty}2^{k\alpha p} \left(  \|f\chi_{k+\ell-1} \|_{q,\omega}+\|f\chi_{k+\ell} \|_{q,\omega}  \right)^p \right)^{\frac{1}{p}}\\
&\le \left(\sum\limits_{k=-\infty}^{\infty}2^{k\alpha p}  \|f\chi_{k+\ell-1} \|_{q,\omega}^p \right)^{\frac{1}{p}} +\left(\sum\limits_{k=-\infty}^{\infty}2^{k\alpha p}  \|f\chi_{k+\ell} \|_{q,\omega}^p \right)^{\frac{1}{p}}\\
&\le (2^{-(\ell-1)\alpha}+2^{-\ell\alpha})\|f\|_{\dot{K}^{\alpha,p}_{q}(\omega)}\lesssim u^{-\alpha}\|f\|_{\dot{K}^{\alpha,p}_{q}(\omega)}.
\end{align*}
Consequently, 
\begin{align*}
\|{{\mathcal H}}_{\Phi,\Omega}f \|_{\dot{K}^{\alpha,p}_{q}(\omega)} \lesssim\|\Omega \|_{L^{q'}(S^{n-1})} . \|f\|_{\dot{K}^{\alpha,p}_{q}(\omega)}.\int\limits_0^{\infty}|\Phi(u^{-1})|u^{1-2n-\frac{\gamma}{q}-\frac{n}{q}-\alpha}du.
\end{align*}
This shows that the operator ${{\mathcal H}}_{\Phi,\Omega}$ is boundedness on the space $\dot{K}^{\alpha,p}_{q}(\omega)$ and $\|{{\mathcal H}}_{\Phi,\Omega}\|_{\dot{K}^{\alpha,p}_{q}(\omega)\to \dot{K}^{\alpha,p}_{q}(\omega)}\lesssim \mathcal{C}_2\|\Omega\|_{L^{q'}(S^{n-1})}$.

(ii) Now, we will give the proof for part (ii) of the theorem.
For $m\in\mathbb{Z}$, we choose $m$ sufficiently large  such that $\alpha+\frac{1}{2^m}\ne 0$.  Let us choose the functions
\begin{align*}
f_{m}(x)=\begin{cases} 0,& \text{if } |x|<1,\\|x|^{-\alpha-\frac{\gamma}{q}-\frac{n}{q}-\frac{1}{2^m}} |\Omega(x')|^{q'-2}\overline{\Omega}(x'),& \text{if } |x|\ge 1.   \end{cases}
\end{align*}
By similar argument as in \cite{CDH2016}, one can show that $f_m\in \dot{K}^{\alpha,p}_{q}(\omega)$. But, for convenience to the reader, we provide details for the proof here. First, we remark that for $k\in\mathbb Z$, $k\geq 0$, we have
\begin{align*}
\|f_{m}\chi_k\|_{q,\omega}&\le \left(\int\limits_{C_k}\left||x|^{-\alpha-\frac{\gamma}{q}-\frac{n}{q}-\frac{1}{2^m}} |\Omega(x')|^{q'-2}\Omega(x')\chi_k(x)\right|^q\omega(x)dx \right)^{\frac{1}{q}}\\
&\le \left(\int\limits_{C_k}\int\limits_{S^{n-1}} |rx'|^{-\alpha q-\gamma-n-\frac{q}{2^m}}|\Omega(x')|^{q(q'-1)}\omega(rx')r^{n-1}d\sigma(x')dr \right)^{\frac{1}{q}}\\
&\lesssim \left(\int\limits_{C_k} r^{-\alpha q-\frac{q}{2^m}-1}\left( \int\limits_{S^{n-1}} |\Omega(x')|^{q(q'-1)}\omega(x')d\sigma(x')\right) dr \right)^{\frac{1}{q}}\\
&\lesssim\left(\int\limits_{C_k}r^{-\alpha q -\frac{q}{2^m}-1}dr\right)^{\frac{1}{q}}\|\Omega\|_{L^{q'}(S^{n-1},\omega(x') d\sigma(x'))}^{\frac{q'}{q}}\\
&=2^{-k(\frac{1}{2^m}+\alpha)}\left|\frac{2^{q(\frac{1}{2^m}+\alpha)}-1}{(\frac{1}{2^m}+\alpha)q} \right|^{\frac{1}{q}}\|\Omega\|_{L^{q'}(S^{n-1},\omega(x') d\sigma(x'))}^{\frac{q'}{q}}.
\end{align*}
It is obvious that for $k<0$, then $\|f_m\chi_k \|_{q,\omega}=0$. Therefore, 
\begin{align*}
\|f_{m} \|_{\dot{K}^{\alpha,p}_{q}(\omega)}&= \left(\sum\limits_{k=-\infty}^{\infty}2^{k\alpha p}\|f_{m}\chi_k \|_{q,\omega}^p \right)^{\frac{1}{p}}\\
&\le\left(\sum\limits_{k=0}^{\infty}2^{k\alpha p} \left(2^{-k(\frac{1}{2^m}+\alpha)}\left|\frac{2^{q(\frac{1}{2^m}+\alpha)}-1}{(\frac{1}{2^m}+\alpha)q} \right|^{\frac{1}{q}} \|\Omega\|_{L^{q'}(S^{n-1},\omega(x') d\sigma(x'))}^{\frac{q'}{q}} \right)^p \right)^{\frac{1}{p}}\\
&\le \left|\frac{2^{q(\frac{1}{2^m}+\alpha)}-1}{(\frac{1}{2^m}+\alpha)q} \right|^{\frac{1}{q}}\|\Omega\|_{L^{q'}(S^{n-1},\omega(x') d\sigma(x'))}^{\frac{q'}{q}}\left(\sum\limits_{k=0}^{\infty}2^{k\alpha p} \left(2^{-k(\frac{1}{2^m}+\alpha)}  \right)^p \right)^{\frac{1}{p}}\\
&\le \left|\frac{2^{q(\frac{1}{2^m}+\alpha)}-1}{(\frac{1}{2^m}+\alpha)q} \right|^{\frac{1}{q}}\|\Omega\|_{L^{q'}(S^{n-1},\omega(x') d\sigma(x'))}^{\frac{q'}{q}}\left(\sum\limits_{k=0}^{\infty}2^{-\frac{kp}{2^m}}\right)^{\frac{1}{p}}<\infty.
\end{align*}
Now, it is easy to see that
\begin{align*}
&\mathcal{H}_{\Phi,\Omega}f_m\\
&=\begin{cases} 0,&\text{if } |x|<1,\\|x|^{-\alpha-\frac{\gamma}{q}-\frac{n}{q}-\frac{1}{2^m}} \int_{S(x)}\left(\int_{S^{n-1}}\Phi(u^{-1})u^{1-2n-\frac{\gamma}{q}-\frac{n}{q}-\frac{1}{2^m}} |\Omega(y')|^{q'}d\sigma(y')  \right)du,&\text{if } |x|\ge 1,   \end{cases}
\end{align*}
where $S(x)=\{u\in(0,\infty): ||x|uy'|\ge 1 \}$. For $k\in \mathbb Z$ such that $k\geq 1$, let
$$ S_k=\left\{u\in(0,\infty): |u|\ge \frac{1}{2^{k-1}}\right\}.$$
It is clear that the sequence $\{S_k \}_{k\ge 0}$ is increasing and tends to $(0,\infty)$. Let $1\le m\le k$. Then,  for all $x\in C_k$,  there exits a measurable subset $A$ of $(0,\infty)$ with $|A|=0$ such that 
\begin{align*}
S(x) \supset S_m\backslash A.
\end{align*}
 Because, for each $k\le 0$, $\mathcal{H}_{\Phi,\Omega}f_m\chi_k=0$, so we have
\begin{align*}
&\|\mathcal{H}_{\Phi,\Omega}f_m\chi_k\|_{q,\omega}\\
&=\left(\int_{C_k}\left||x|^{-\alpha-\frac{\gamma}{q}-\frac{n}{q}-\frac{1}{2^m}} \int_{S_k}\left(\int_{S^{n-1}}\Phi(u^{-1})u^{1-2n-\frac{\gamma}{q}-\frac{n}{q}-\frac{1}{2^m}} |\Omega(y')|^{q'}d\sigma(y')  \right)du  \right|^q\omega(x)dx\right)^{\frac{1}{q}}\\
&\ge \left(\int_{C_k}|x|^{-\alpha-\frac{\gamma}{q}-\frac{n}{q}-\frac{1}{2^m}}\omega(x)\left( \int_{S_m}\left(\int_{S^{n-1}}|\Phi(u^{-1})|u^{1-2n-\frac{\gamma}{q}-\frac{n}{q}-\frac{1}{2^m}} |\Omega(y')|^{q'}d\sigma(y')  \right)du \right)^q dx\right)^{\frac{1}{q}}\\
&\ge \left(\int_{S_m}|\Phi(u^{-1})|u^{1-2n-\frac{\gamma}{q}-\frac{n}{q}-\frac{1}{2^m}}du\right)\left(\int_{C_k}|x|^{-\alpha q-\gamma-n-\frac{q}{2^m}}\omega(x)dx\right)^{\frac{1}{q}}  \left(\int_{S^{n-1}}  |\Omega(y')|^{q'}d\sigma(y')   \right)\\
&=\int_{S_m}|\Phi(u^{-1})|u^{1-2n-\frac{\gamma}{q}-\frac{n}{q}-\frac{1}{2^m}}du\|f'_m\chi_k \|_{q,\omega}\|\Omega \|^{q'}_{L^{q'}(S^{n-1})},
\end{align*}
where
\begin{align*}
f'_{m}(x)=\begin{cases} 0, &\text{if } |x|<1,\\|x|^{-\alpha-\frac{\gamma}{q}-\frac{n}{q}-\frac{1}{2^m}} ,&\text{if } |x|\ge 1.   \end{cases}
\end{align*}
It is clear that $\|f'_m\chi_{\mathbb{R}^n\backslash B(0,1)} \|_{q,\omega}=0$ for all $k\le 0$. Therefore,
\begin{align*}
&\|\mathcal{H}_{\Phi,\Omega}f_m \|_{\dot{K}^{\alpha,p}_{q}(\omega)}\\
&\ge \left(\sum\limits_{k=-\infty}^{\infty}2^{k\alpha p}   \left(\int_{S_m}|\Phi(u^{-1})|u^{1-2n-\frac{\gamma}{q}-\frac{n}{q}-\frac{1}{2^m}}du\|f'_m\chi_k \|_{q,\omega}\|\Omega \|^{q'}_{L^{q'}(S^{n-1})} \right)^p \right)^{\frac{1}{p}}\\
&\ge \|\Omega \|^{q'}_{L^{q'}(S^{n-1})} \left(\sum\limits_{k=-\infty}^{\infty}2^{k\alpha p}  \|f'_m\chi_k \|_{q,\omega}^p\right)^{\frac{1}{p}} \left(\int_{S_m}|\Phi(u^{-1})|u^{1-2n-\frac{\gamma}{q}-\frac{n}{q}-\frac{1}{2^m}}du \right)\\
&\ge \|\Omega \|^{q'}_{L^{q'}(S^{n-1})} \left(\sum\limits_{k=m}^{\infty}2^{k\alpha p}  2^{-kp(\frac{1}{2^m}+\alpha)}\left|\frac{2^{q(\frac{1}{2^m}+\alpha)}-1}{(\frac{1}{2^m}+\alpha)q} \right|^{\frac{p}{q}}\right)^{\frac{1}{p}}\mathcal{\mathcal C}_2(m)\\
&\ge \|\Omega \|^{q'}_{L^{q'}(S^{n-1})}2^{-\frac{m}{2^m}} \left(\sum\limits_{k=0}^{\infty}2^{-\frac{kp}{2^m}}\right)^{\frac{1}{p}}\left|\frac{2^{q(\frac{1}{2^m}+\alpha)}-1}{(\frac{1}{2^m}+\alpha)q} \right|^{\frac{1}{q}}\mathcal{C}_2(m),
\end{align*}
where $\mathcal{C}_2(m):=\displaystyle\int_{S_m}|\Phi(u^{-1})|u^{1-2n-\frac{\gamma}{q}-\frac{n}{q}-\frac{1}{2^m}}du$. Now, since the operator $\mathcal{H}_{\Phi,\Omega}$ is bounded on the space $\dot{K}^{\alpha,p}_{q}(\omega)$, we yield
\begin{align*} 
&\|\mathcal{H}_{\Phi,\Omega} \|_{\dot{K}^{\alpha,p}_{q}(\omega)\to \dot{K}^{\alpha,p}_{q}(\omega)}\ge \frac{\|\mathcal{H}_{\Phi,\Omega}f_m \|_{\dot{K}^{\alpha,p}_{q}(\omega)}}{\|f_m\|_{\dot{K}^{\alpha,p}_{q}(\omega)}}\\
&\ge \frac{\|\Omega \|^{q'}_{L^{q'}(S^{n-1})}2^{-\frac{m}{2^m}} \left(\sum\limits_{k=0}^{\infty} 2^{-\frac{kp}{2^m}}\right)^{\frac{1}{p}}\left|\frac{2^{q(\frac{1}{2^m}+\alpha)}-1}{(\frac{1}{2^m}+\alpha)q} \right|^{\frac{1}{q}} \mathcal{C}_2(m)}{\left|\frac{2^{q(\frac{1}{2^m}+\alpha)}-1}{(\frac{1}{2^m}+\alpha)q} \right|^{\frac{1}{q}}\|\Omega\|_{L^{q'}(S^{n-1},\omega(x') d\sigma(x'))}^{\frac{q'}{q}}\left(\sum\limits_{k=0}^{\infty}2^{-\frac{kp}{2^m}}   \right)^{\frac{1}{p}}}\\
&\ge \frac{\|\Omega \|^{q'}_{L^{q'}(S^{n-1})}}{\|\Omega \|^{\frac{q'}{q}}_{L^{q'}(S^{n-1},\omega(x') d\sigma(x'))}} .2^{-\frac{m}{2^m}} \int_{S_m}|\Phi(u^{-1})|u^{1-2n-\frac{\gamma}{q}-\frac{n}{q}-\frac{1}{2^m}}du.
\end{align*}
Thus, letting $m\to \infty$,  by the Lebesgue dominated convergence theorem we obtain
\begin{align*}
\|\mathcal{H}_{\Phi,\Omega} \|_{\dot{K}^{\alpha,p}_{q}(\omega)\to \dot{K}^{\alpha,p}_{q}(\omega)} &\gtrsim \int_{0}^\infty|\Phi(u^{-1})|u^{1-2n-\frac{\gamma}{q}-\frac{n}{q}}du.\frac{\|\Omega \|^{q'}_{L^{q'}(S^{n-1})}}{\|\Omega \|^{\frac{q'}{q}}_{L^{q'}(S^{n-1},\omega(x') d\sigma(x'))}} . 
\end{align*}
\end{proof}
By Theorem \ref{Herz}, we also have the following useful corollary.
\begin{corollary}
Let  $1\le p,q<\infty$ and $\Omega\in L^{q'}(S^{n-1})$, $\omega(x)=|x|^\gamma$. Let $\Phi$ be a nonnegative radial function. Then, ${\mathcal H}_{\Phi,\Omega}$  is a bounded operator on $\dot{K}^{\alpha,p}_{q}(\omega)$ if and only if 
$$\mathcal{C}_{2.1}=\int_{0}^\infty|\Phi(t^{-1})|\;t^{1-2n-\frac{\gamma}{q}-\frac{n}{q}}dt<\infty.$$
 Moreover,
\begin{align*}
\|{\mathcal H}_{\Phi,\Omega} \|_{\dot{K}^{\alpha,p}_{q}(\omega)}\simeq\mathcal{C}_{2.1}.\|\Omega\|_{L^{q'}(S^{n-1})}.
\end{align*}
\end{corollary}
Next, we also give the boundedness and bound of the rough Hausdorff operator on the weighted Morrey-Herz spaces. 
\begin{theorem}\label{MorreyHerz}
Let $1\le q<\infty, 0<p<\infty, \gamma\in\mathbb{R}, \lambda>0$, and $\Omega\in L^{q'}(S^{n-1})$.
\\
{\rm (i)} If $\omega(x')\ge c>0$ for all $x'\in S^{n-1}$ and
\begin{align*}
\mathcal C_3=\int\limits_0^{\infty}\frac{|\Phi(t)|}{t^{1-\frac{\gamma}{q}-\frac{n}{q}+\lambda-\alpha}}dt<\infty,
\end{align*}
then ${{\mathcal H}}_{\Phi,\Omega}$ is a bounded operator on $M\dot{K}^{\alpha,\lambda}_{p,q}(\omega)$. Moreover,
\begin{align*}
\|{\mathcal H}_{\Phi,\Omega} \|_{M\dot{K}^{\alpha,\lambda}_{p,q}(\omega)}\lesssim \mathcal{C}_3\|\Omega\|_{L^{q'}(S^{n-1})}.
\end{align*}
{\rm (ii)} Conversely, suppose $\Omega\in L^{q'}(S^{n-1}, \omega(x')d\sigma(x'))$ and $\Phi$ is a real function with a constant sign in $\mathbb{R}^n$. Then, if ${\mathcal H}_{\Phi,\Omega}$ is bounded on the $ M\dot{K}^{\alpha,\lambda}_{p,q}(\omega)$, then $\mathcal{C}_3<\infty$. Furthermore,
\begin{align*}
\|{{\mathcal H}}_{\Phi,\Omega} \|_{M\dot{K}^{\alpha,\lambda}_{p,q}(\omega)\to M\dot{K}^{\alpha,\lambda}_{p,q}(\omega)}\ge \mathcal C_3.\frac{\|\Omega \|^{q'}_{L^{q'}(S^{n-1})}}{\|\Omega \|^{\frac{q'}{q}}_{L^{q'}(S^{n-1},\omega(x') d\sigma(x'))}}.
\end{align*}
\end{theorem}

\begin{proof}
(i) From the Minkowski inequality and changing variable $u=xt^{-1}$, we obtain
\begin{align}\label{eqa6}
\|{{\mathcal H}}_{\Phi,\Omega}f\chi_k \|_{q,\omega}&\le\int\limits_0^{\infty}\frac{|\Phi(t)|}{t}\left( \int\limits_{C_k}\left|  \int\limits_{S^{n-1}}\Omega(y')f(|x|t^{-1}y')d\sigma(y') \right|^q\omega(x) dx \right)^{\frac{1}{q}}dt\notag\\
&=\int\limits_0^{\infty}\frac{|\Phi(t)|}{t^{1-\frac{\gamma}{q}-\frac{n}{q}}}\left( \int\limits_{\frac{1}{t}C_k}\left|  \int\limits_{S^{n-1}}\Omega(y')f(|u|y')d\sigma(y') \right|^q\omega(u) du \right)^{\frac{1}{q}}dt.
\end{align}

By \eqref{eq6a1} and \eqref{eqa6}, it follows that
\begin{align*}
&\|{{\mathcal H}}_{\Phi,\Omega}f\chi_k \|_{q,\omega}\notag\\
&\le\int\limits_0^{\infty}\frac{|\Phi(t)|}{t^{1-\frac{\gamma}{q}-\frac{n}{q}}}\left( \int\limits_{\frac{1}{t}C_k}\left|\left(\int\limits_{S^{n-1}}|f(|u|y')|^qd\sigma(y') \right)^{\frac{1}{q}} \|\Omega \|_{L^{q'}(S^{n-1})} \right|^q\omega(u) du \right)^{\frac{1}{q}}dt\notag\\
&\le \|\Omega \|_{L^{q'}(S^{n-1})} \int\limits_0^{\infty}\frac{|\Phi(t)|}{t^{1-\frac{\gamma}{q}-\frac{n}{q}}}\mathcal{J'}(t)^{\frac{1}{q}}dt,
\end{align*}
where $\mathcal{J'}(t):=\displaystyle\int\limits_{\frac{1}{t}C_k} \left(\int\limits_{S^{n-1}}|f(|u|y')|^qd\sigma(y') \right) \omega(u) du$. By the similar estimate as $\mathcal{J}(u)$, we also have
\begin{align*}
\mathcal{J'}(t)\lesssim\|f\chi_{\frac{1}{t}C_k} \|_{q,\omega}^q.
\end{align*}

Note that for each $t\in (0, \infty)$, we  can find an integer number $\ell=\ell(t)$ such that $2^{\ell-1}<\dfrac{1}{t}\le 2^{\ell}$. This implies that $\frac{1}{t}C_k$ is a subset of $C_{k+\ell-1}\cup C_{k+\ell}$. Thus, we obtain
\begin{align*}
\|f\chi_{\frac{1}{t}C_k}\|_{q,\omega}\le \|f\chi_{k+\ell-1} \|_{q,\omega}+\|f\chi_{k+\ell} \|_{q,\omega}.
\end{align*}
Hence, 
\begin{align}\label{eq312a}
\|{{\mathcal H}}_{\Phi,\Omega}f\chi_k \|_{q,\omega}\lesssim \|\Omega \|_{L^{q'}(S^{n-1})} \int\limits_0^{\infty}\frac{|\Phi(t)|}{t^{1-\frac{\gamma}{q}-\frac{n}{q}}}\left( \|f\chi_{k+\ell-1} \|_{q,\omega}+\|f\chi_{k+\ell} \|_{q,\omega} \right) dt.
\end{align}
We consider two case as follows.\\
Case 1: $1\le p< \infty$. Then, we get
\begin{align*}
&\|{{\mathcal H}}_{\Phi,\Omega}f \|_{M\dot{K}^{\alpha,\lambda}_{p,q}(\omega)}\\
&= \sup\limits_{k_0\in\mathbb{Z}}2^{-k_0\lambda}\left(\sum\limits_{k=-\infty}^{k_0}2^{k\alpha p}\|{{\mathcal H}}_{\Phi,\Omega}f(x)\chi_k(x) \|_{q,\omega}^p \right)^{\frac{1}{p}}\\
&\lesssim \sup\limits_{k_0\in\mathbb{Z}}2^{-k_0\lambda}\left(\sum\limits_{k=-\infty}^{k_0}2^{k\alpha p}\left(\|\Omega \|_{L^{q'}(S^{n-1})} \int\limits_0^{\infty}\frac{|\Phi(t)|}{t^{1-\frac{\gamma}{q}-\frac{n}{q}}}\left( \|f\chi_{k+\ell-1} \|_{q,\omega}+\|f\chi_{k+\ell} \|_{q,\omega} \right) dt \right)^p \right)^{\frac{1}{p}}\\
&\lesssim \|\Omega \|_{L^{q'}(S^{n-1})} \int\limits_0^{\infty}\frac{|\Phi(t)|}{t^{1-\frac{\gamma}{q}-\frac{n}{q}}}\sup\limits_{k_0\in\mathbb{Z}}2^{-k_0\lambda}\left(\sum\limits_{k=-\infty}^{k_0}2^{k\alpha p}\left( \|f\chi_{k+\ell-1} \|_{q,\omega}+\|f\chi_{k+\ell} \|_{q,\omega}  \right)^p \right)^{\frac{1}{p}}dt.
\end{align*}
It is clear that
\begin{align*}
&\sup\limits_{k_0\in\mathbb{Z}}2^{-k_0\lambda}\left(\sum\limits_{k=-\infty}^{k_0}2^{k\alpha p} \left(  \|f\chi_{k+\ell-1} \|_{q,\omega}+\|f\chi_{k+\ell} \|_{q,\omega}  \right)^p \right)^{\frac{1}{p}}\\
&\le \sup\limits_{k_0\in\mathbb{Z}}2^{-k_0\lambda}\left(\sum\limits_{k=-\infty}^{k_0}2^{k\alpha p}  \|f\chi_{k+\ell-1} \|_{q,\omega}^p \right)^{\frac{1}{p}} +\sup\limits_{k_0\in\mathbb{Z}}2^{-k_0\lambda}\left(\sum\limits_{k=-\infty}^{k_0}2^{k\alpha p}  \|f\chi_{k+\ell} \|_{q,\omega}^p \right)^{\frac{1}{p}}\\
&\lesssim 2^{\ell(\lambda-\alpha)}\|f\|_{M\dot{K}_{p,q}^{\alpha,\lambda}(\omega)} \lesssim \left(\frac{1}{t} \right)^{\lambda-\alpha}\|f\|_{M\dot{K}_{p,q}^{\alpha,\lambda}(\omega)}.
\end{align*}
Consequently, 
\begin{align*}
\|{{\mathcal H}}_{\Phi,\Omega}f \|_{M\dot{K}^{\alpha,\lambda}_{p,q}(\omega)}&\lesssim \|\Omega \|_{L^{q'}(S^{n-1})} \int\limits_0^{\infty}\frac{|\Phi(t)|}{t^{1-\frac{\gamma}{q}-\frac{n}{q}}}\left(\frac{1}{t}\right)^{\lambda-\alpha}\|f\|_{M\dot{K}_{p,q}^{\alpha,\lambda}(\omega)} dt\\
&\lesssim \|\Omega \|_{L^{q'}(S^{n-1})}\|f\|_{M\dot{K}_{p,q}^{\alpha,\lambda}(\omega)}. \int\limits_0^{\infty}\frac{|\Phi(t)|}{t^{1-\frac{\gamma}{q}-\frac{n}{q}+\lambda-\alpha}} dt.
\end{align*}
Case 2: $0<p<1$. It follows from the definition of weighted Morrey-Herz space that
\begin{align*}
\|f\chi_k\|_{q,\omega}\le 2^{k(\lambda-\alpha)}\|f\|_{M\dot{K}^{\alpha,\lambda}_{p,q}(\omega)}.
\end{align*}
For all $f\in M\dot{K}^{\alpha,\lambda}_{p,q}(\omega)$, by \eqref{eq312a}, we obtain
\begin{align*}
\|{{\mathcal H}}_{\Phi,\Omega}f\chi_k \|_{q,\omega}\lesssim \|\Omega \|_{L^{q'}(S^{n-1})} \int\limits_0^{\infty}\frac{|\Phi(t)|}{t^{1-\frac{\gamma}{q}-\frac{n}{q}}}\left( \sum\limits_{i=-1,0}2^{(k+\ell+i)(\lambda-\alpha)}\|f\|_{M\dot{K}^{\alpha,\lambda}_{p,q}(\omega)} \right) dt,
\end{align*}
for all $k\in\mathbb Z$. Thus, 
\begin{align*}
&\|{{\mathcal H}}_{\Phi,\Omega}f \|_{M\dot{K}^{\alpha,\lambda}_{p,q}(\omega)}= \sup\limits_{k_0\in\mathbb{Z}}2^{-k_0\lambda}\left(\sum\limits_{k=-\infty}^{k_0}2^{k\alpha p}\|{{\mathcal H}}_{\Phi,\Omega}f(x)\chi_k(x) \|_{q,\omega}^p \right)^{\frac{1}{p}}\\
&\lesssim \|\Omega \|_{L^{q'}(S^{n-1})}\|f\|_{M\dot{K}^{\alpha,\lambda}_{p,q}(\omega)}. \sup\limits_{k_0\in\mathbb{Z}}2^{-k_0\lambda}\times\\
&\times\left(\sum\limits_{k=-\infty}^{k_0}2^{k\alpha p}\left( \int\limits_0^{\infty}\frac{|\Phi(t)|}{t^{1-\frac{\gamma}{q}-\frac{n}{q}}}\left(\sum\limits_{i=-1,0}2^{(k+\ell+i)(\lambda-\alpha)}\right) dt \right)^p\right)^{\frac{1}{p}}.
\end{align*}
Since $2^{\ell-1}<\dfrac{1}{t}\le 2^{\ell}$ and $\lambda>0$, we estimate
\begin{align*}
\sum\limits_{i=-1,0}2^{(k+l+i)(\lambda-\alpha)}\lesssim \left(\frac{1}{t} \right)^{\lambda-\alpha}2^{k(\lambda-\alpha)}\sum\limits_{i=-1,0}2^{i(\lambda-\alpha)}.
\end{align*}
Therefore, 
\begin{align*}
&\|{{\mathcal H}}_{\Phi,\Omega}f \|_{M\dot{K}^{\alpha,\lambda}_{p,q}(\omega)}\\
&\lesssim \|\Omega \|_{L^{q'}(S^{n-1})} \|f \|_{M\dot{K}^{\alpha,\lambda}_{p,q}(\omega)}\sup\limits_{k_0\in\mathbb{Z}}\left(\sum\limits_{k=-\infty}^{k_0}2^{(k-k_0)\lambda p} \right)^{\frac{1}{p}} \left(\int\limits_0^{\infty}\frac{|\Phi(t)|}{t^{1-\frac{\gamma}{q}-\frac{n}{q}+\lambda-\alpha}}dt \right)\sum\limits_{i=-1,0}2^{i(\lambda-\alpha)}\\
&\lesssim \|\Omega \|_{L^{q'}(S^{n-1})} \|f \|_{M\dot{K}^{\alpha,\lambda}_{p,q}(\omega)}\sup\limits_{k_0\in\mathbb{Z}}\left(\sum\limits_{k=-\infty}^{k_0}2^{(k-k_0)\lambda p} \right)^{\frac{1}{p}} \left(\int\limits_0^{\infty}\frac{|\Phi(t)|}{t^{1-\frac{\gamma}{q}-\frac{n}{q}+\lambda-\alpha}}dt \right)\\
&\lesssim \|\Omega \|_{L^{q'}(S^{n-1})}.\|f \|_{M\dot{K}^{\alpha,\lambda}_{p,q}(\omega)}. \int\limits_0^{\infty}\frac{|\Phi(t)|}{t^{1-\frac{\gamma}{q}-\frac{n}{q}+\lambda-\alpha}}dt .
\end{align*}

(ii) Conversely, suppose ${\mathcal H}_{\Phi,\Omega}$ is bounded on the  space $M\dot{K}^{\alpha,\lambda}_{p,q}(\omega)$. Then, let us choose the function
\begin{align*}
f(x)=|x|^{-\alpha-\frac{n}{q}-\frac{\gamma}{q}+\lambda}|\Omega(x')|^{q'-2}\overline{\Omega}(x').
\end{align*}
We have
\begin{align*}
\|f\chi_k\|_{q,\omega}&=\left(\int\limits_{\mathbb{R}^n}||x|^{-\alpha-\frac{n}{q}-\frac{\gamma}{q}+\lambda}|\Omega(x')|^{q'-2}.\overline{\Omega}(x')\chi_k|^q\omega(x)dx \right)^{\frac{1}{q}}\\
&=\left(\int\limits_{C_k}\int\limits_{S^{n-1}}r^{-\alpha q-n-\gamma+\lambda q}|\Omega(x')|^{q'}r^\gamma\omega(x')r^{n-1}d\sigma(x')dr \right)^{\frac{1}{q}}\\
&=\left(\int\limits_{C_k}r^{-\alpha q+\lambda q-1}dr\int\limits_{S^{n-1}}|\Omega(x')|^{q'}\omega(x') d\sigma(x') \right)^{\frac{1}{q}}\\
&=\left(\int\limits_{C_k}r^{-\alpha q+\lambda q-1}dr\right)^{\frac{1}{q}}\|\Omega\|_{L^{q'}(S^{n-1},\omega(x') d\sigma(x'))}^{\frac{q'}{q}}\\
&= \begin{cases} \ln{2}.\|\Omega\|_{L^{q'}(S^{n-1},\omega(x') d\sigma(x'))}^{\frac{q'}{q}},&\text{if } \alpha=\gamma,\\
2^{k(\lambda-\alpha)}\left|\frac{1-2^{-q(\lambda-\alpha)}}{q(\lambda-\alpha)} \right|^{\frac{1}{q}}\|\Omega\|_{L^{q'}(S^{n-1},\omega(x') d\sigma(x'))}^{\frac{q'}{q}}, &\text{if } \alpha\not =\gamma.
\end{cases}
\end{align*}
Therefore, an easy computation shows that
\begin{align*}
&\|f\|_{M\dot{K}^{\alpha,\lambda}_{p,q}(\omega)}=\sup\limits_{k_0\in \mathbb{Z}}2^{-k_0\lambda}\left(\sum\limits_{k=-\infty}^{k_0}2^{k\alpha p}\|f\chi_k \|_{q,\omega}^p \right)^{\frac{1}{p}}\\
&\lesssim \|\Omega\|_{L^{q'}(S^{n-1},\omega(x') d\sigma(x'))}^{\frac{q'}{q}}\sup\limits_{k_0\in \mathbb{Z}}2^{-k_0\lambda}\left(\sum\limits_{k=-\infty}^{k_0}2^{k\alpha p}\left(2^{k(\lambda-\alpha)} \right)^p    \right)^{\frac{1}{p}}\\
&\lesssim \|\Omega\|_{L^{q'}(S^{n-1},\omega(x') d\sigma(x'))}^{\frac{q'}{q}}\sup\limits_{k_0\in \mathbb{Z}}2^{-k_0\lambda}\left(\sum\limits_{k=-\infty}^{k_0}2^{k\lambda p}\right)^{\frac{1}{p}}<\infty.\\
\end{align*}
On the other hand, we also get
\begin{align*}
{{\mathcal H}}_{\Phi,\Omega}f(x)
&=\int\limits_0^{\infty}\left(\int\limits_{S^{n-1}}\frac{\Phi(t)}{t}\Omega(y')f(|x|t^{-1}y') d\sigma(y')\right)dt\\
&=|x|^{-\alpha-\frac{n}{q}-\frac{\gamma}{q}+\lambda}.\int\limits_0^{\infty}\frac{\Phi(t)}{t}t^{\alpha+\frac{n}{q}+\frac{\gamma}{q}-\lambda}\left(\int\limits_{S^{n-1}}|\Omega(y')|^{q'} d\sigma(y')\right)dt\\
&=|x|^{-\alpha-\frac{n}{q}-\frac{\gamma}{q}+\lambda}.\|\Omega\|_{L^{q'}(S^{n-1})}^{q'}\int\limits_0^{\infty}\frac{\Phi(t)}{t^{1-\frac{\gamma}{q}-\frac{n}{q}+\lambda-\alpha}}dt.
\end{align*}
Hence, it immediately follows that
\begin{align*}
\|{{\mathcal H}}_{\Phi,\Omega}f \|_{M\dot{K}^{\alpha,\lambda}_{p,q}(\omega)}\simeq \||x|^{-\alpha-\frac{n}{q}-\frac{\gamma}{q}+\lambda}\|_{M\dot{K}^{\alpha,\lambda}_{p,q}(\omega)}.\|\Omega\|_{L^{q'}(S^{n-1})}^{q'}\int\limits_0^{\infty}\frac{\Phi(t)}{t^{1-\frac{\gamma}{q}-\frac{n}{q}+\lambda-\alpha}}dt.
\end{align*}
Therefore, 
\begin{align*}
\|{{\mathcal H}}_{\Phi,\Omega} \|_{M\dot{K}^{\alpha,\lambda}_{p,q}(\omega)\to M\dot{K}^{\alpha,\lambda}_{p,q}(\omega)}&\ge \frac{\|{{\mathcal H}}_{\Phi,\Omega}f \|_{M\dot{K}^{\alpha,\lambda}_{p,q}(\omega)}}{\|f\|_{M\dot{K}^{\alpha,\lambda}_{p,q}(\omega)}}\\
&\gtrsim \frac{\int\limits_0^{\infty}\frac{\Phi(t)}{t^{1-\frac{\gamma}{q}-\frac{n}{q}+\lambda-\alpha}}dt\|\Omega\|_{L^{q'}(S^{n-1})}^{q'}\||x|^{-\alpha-\frac{n}{q}-\frac{\gamma}{q}+\lambda}\|_{M\dot{K}^{\alpha,\lambda}_{p,q}(\omega)}}{\|\Omega \|^{\frac{q'}{q}}_{L^{q'}(S^{n-1},\omega(x')d\sigma(x'))}\||x|^{-\alpha-\frac{n}{q}-\frac{\gamma}{q}+\lambda} \|_{M\dot{K}^{\alpha,\lambda}_{p,q}(\omega)}}\\
&\gtrsim \int\limits_0^{\infty}\frac{\Phi(t)}{t^{1-\frac{\gamma}{q}-\frac{n}{q}+\lambda-\alpha}}dt.\frac{\|\Omega\|^{q'}_{L^{q'}(S^{n-1})}}{\|\Omega \|^{\frac{q'}{q}}_{L^{q'}(S^{n-1},\omega(x')d\sigma(x'))}}.
\end{align*}
This ends the proof of theorem.
\end{proof}
By Theorem \ref{MorreyHerz}, we have the following useful corollary when $\omega$ is a power weight function and $\Phi$ is a nonnegative function.
\begin{corollary}
Let  $1\le q<\infty,0\le p<\infty,\gamma\in\mathbb{R},\lambda>0$. Suppose  $\Omega\in L^{q'}(S^{n-1})$, $\omega(x)=|x|^\gamma$, and $\Phi$ is a nonnegative radial function. Then, ${\mathcal H}_{\Phi,\Omega}$  is a bounded operator on $M\dot{K}^{\alpha,\lambda}_{p,q}(\omega)$ if and only if 
$$\mathcal{C}_{3.1}=\int\limits_0^{\infty}\frac{\Phi(t)}{t^{1-\frac{\gamma}{q}-\frac{n}{q}+\lambda-\alpha}}dt<\infty.$$
 Moreover,
\begin{align*}
\|{\mathcal H}_{\Phi,\Omega} \|_{M\dot{K}^{\alpha,\lambda}_{p,q}(\omega)}\simeq\mathcal{C}_{3.1}.\|\Omega\|_{L^{q'}(S^{n-1})}.
\end{align*}
\end{corollary}
Next, we will give the boundedness of the commutator of rough Hausdorff operator on weighted spaces of Morrey-Herz type with their symbols $b$ belonging to Lipschitz space $Lip^{\beta} (\mathbb{R}^n)$ ($0<\beta\leq 1$). Before stating our next results, we want to give the following useful inequality
\begin{align}\label{ceq1}
|b(x)-b(|x|t^{-1}y')|&\le \|b\|_{Lip^{\beta}}.\|x-|x|t^{-1}y' \|^{\beta}\notag\\
& =\|b\|_{Lip^{\beta}}. |x|^{\beta}(1+t^{-1})^{\beta},\;\forall t >0,\;y'\in S^{n-1}.
\end{align}
\begin{theorem}
Let $1\leq p<\infty$, $\omega_1, \omega_2\in\mathcal{W}_\gamma$ for $\gamma>-n$, and 
$b\in Lip^{\beta} (\mathbb{R}^n)$ for $0<\beta\le 1$.  Let  $\Omega\in L^{p'}(S^{n-1})$ and $ \omega_2(x')\ge c>0$ for all $x'\in S^{n-1}$.   Suppose that  $\lambda_1=\lambda-\frac{\beta p}{n+\gamma}>0$.  Then, if
\begin{equation}
\mathcal C_4=\int\limits_0^{\infty}\frac{|\Phi(t)|}{t^{1+(\gamma+n)\frac{\lambda_1-1}{p}}(1+t^{-1})^{-\beta} } dt  <\infty, \notag
\end{equation}
 the commutator  ${{\mathcal H}}_{\Phi,\Omega}^b$  is a bounded operator from $\dot{B}^{p,\lambda_1}(\omega_1,\omega_2)$ to $\dot{B}^{p,\lambda}(\omega_1,\omega_2)$.
\end{theorem}

\begin{proof}
It is easy to see that  for any $x\in B(0,R)$, then $|x|^\beta\leq |B(0,R)|^{\frac{\beta}{n}}$. It is  also important to note that  $\omega_2(B(0, R))\simeq |B(0,R)|^{\frac{n+\gamma}{n}}$ for all $\gamma>-n$.
From this and by \eqref{ceq1} above, for all $f\in {\dot{B}^{p,\lambda_1}(\omega_1,\omega_2)}$ we have
\begin{align*}
&\|{{\mathcal H}}^b_{\Phi,\Omega}f \|_{\dot{B}^{p,\lambda}(\omega_1,\omega_2)}=\sup\limits_{R>0}\left(\frac{1}{\omega_2(B(0,R))^{\lambda}}\int\limits_{B(0,R)}\left|{\mathcal H}^b_{\Phi,\Omega}f    \right|^p\omega_1(x)dx \right)^{\frac{1}{p}}\\
&\le\|b\|_{Lip^{\beta}}\sup\limits_{R>0}\left(\frac{1}{\omega_2(B(0,R))^{\lambda}}\int\limits_{B(0,R)}\left|\int\limits_0^{\infty}\int\limits_{S^{n-1}}\frac{\Phi(t)}{t(1+t^{-1})^{-\beta}}  \Omega(y')f(|x|t^{-1}y')\right.\right.\\
&\left.\left.\Big.\times |x|^{\beta} d\sigma(y') dt   \right|^p\omega_1(x)dx \right)^{\frac{1}{p}}\\
&\lesssim\|b\|_{Lip^{\beta}}\sup\limits_{R>0}\left(\frac{1}{\omega_2(B(0,R))^{\lambda_1}}\int\limits_{B(0,R)}\left|\int\limits_0^{\infty}\int\limits_{S^{n-1}}\frac{\Phi(t)}{t(1+t^{-1})^{-\beta}}\Omega(y')f(|x|t^{-1}y')\right.\right.\\
&\left.\left.\Big.\times d\sigma(y') dt   \right|^p\omega_1(x)dx \right)^{\frac{1}{p}},
\end{align*}
where $\lambda_1=\lambda-\frac{\beta p}{n+\gamma}$.
Now, using the Minkowski inequality and changing variable $u=xt^{-1}$, we get
\begin{align*}
&\|{{\mathcal H}}_{\Phi,\Omega}f \|_{\dot{B}^{p,\lambda}(\omega_1,\omega_2)}\\
&\lesssim\|b\|_{Lip^{\beta}}\sup\limits_{R>0}\int\limits_0^{\infty}\frac{|\Phi(t)|}{t(1+t^{-1})^{-\beta}}\left(\frac{1}{\omega_2(B(0,R))^{\lambda_1}}\int\limits_{B(0,R)}\left|\int\limits_{S^{n-1}}\Omega(y')f(|x|t^{-1}y') d\sigma(y')    \right|^p \right.\\
&\left.\Big.\times\omega_1(x)dx \right)^{\frac{1}{p}}dt\\
&\lesssim\|b\|_{Lip^{\beta}}\sup\limits_{R>0}\int\limits_0^{\infty}\frac{|\Phi(t)|}{t^{1-\frac{\gamma}{p}-\frac{n}{p}} (1+t^{-1})^{-\beta}}\left(\frac{1}{\omega_2(B(0,R))^{\lambda_1}}\int\limits_{B(0,t^{-1}R)}\left|\int\limits_{S^{n-1}}\Omega(y')f(|u|y') d\sigma(y')    \right|^p\right.\\
&\left.\Big.\times \omega_1(u)du \right)^{\frac{1}{p}}dt.
\end{align*}
It follows from \eqref{eq2} that
\begin{align}\label{ceq2}
&\|{{\mathcal H}}_{\Phi,\Omega}f \|_{\dot{B}^{p,\lambda}(\omega_1,\omega_2)}\lesssim\|b\|_{Lip^{\beta}}\|\Omega \|_{L^{p'}(S^{n-1})}\sup\limits_{R>0}\int\limits_0^{\infty}\frac{|\Phi(t)|}{t^{1-\frac{\gamma}{p}-\frac{n}{p}} (1+t^{-1})^{-\beta}}\mathcal{F}(t) dt,
\end{align}
where   $\mathcal{F}(t):=\left(\dfrac{1}{\omega_2(B(0,R))^{\lambda_1}}\displaystyle\int\limits_{B(0,t^{-1}R)}\left(\displaystyle\int\limits_{S^{n-1}}|f(|u|y')|^pd\sigma(y') \right) \omega_1(u)du \right)^{\frac{1}{p}}$.
Now, we put $u=rx'$, so
\begin{align}\label{ceq3}
\mathcal{F}(t)&=\left(\frac{1}{\omega_2(B(0,R))^{\lambda_1}}\int\limits_{B(0,R)}\int\limits_{S^{n-1}}\left(\int\limits_{S^{n-1}}|f(|rx'|y')|^pd\sigma(y') \right)d\sigma(x') \omega_1(rx')r^{n-1}dr \right)^{\frac{1}{p}}\notag\\
&=\omega(S^{n-1})^{\frac{1}{p}}\left(\frac{1}{\omega_2(B(0,R))^{\lambda_1}}\int\limits_{B(0,R)}r^{\gamma+ n-1}\left(\int\limits_{S^{n-1}}|f(|r|y')|^pd\sigma(y') \right) dr \right)^{\frac{1}{p}}\notag\\
&\lesssim\left(\frac{1}{\omega_2(B(0,R))^{\lambda_1}}\int\limits_{B(0,R)}r^{\gamma+ n-1}\left(\int\limits_{S^{n-1}}|f(|r|y')|^pd\sigma(y') \right) dr \right)^{\frac{1}{p}}.
\end{align}
Note that we have $\dfrac{1}{\omega_2(B(0,R))^{\lambda_1}}=\dfrac{1}{t^{(\gamma+ n)\lambda_1}\omega_2(B(0,{t^{-1}R}))^{\lambda_1}}$. Hence,  by \eqref{ceq2}, \eqref{ceq3} and the condition $\omega_1(x')>c>0$ for all $x'\in S^{n-1}$, we obtain
\begin{align*}
&\|{{\mathcal H}}_{\Phi,\Omega}f \|_{\dot{B}^{p,\lambda}(\omega_1,\omega_2)}\\
&\lesssim\|b\|_{Lip^{\beta}}\|\Omega \|_{L^{p'}(S^{n-1})}\sup\limits_{R>0}\int\limits_0^{\infty}\frac{|\Phi(t)|}{t^{1-\frac{\gamma}{p}-\frac{n}{p}} (1+t^{-1})^{-\beta}}\left(\frac{1}{t^{(\gamma+ n)\lambda_1}\omega_2(B(0,{t^{-1}R}))^{\lambda_1}}\right.\\
&\left.\times\int\limits_{B(0,t^{-1}R)}r^{\gamma+ n-1}\left(\int\limits_{S^{n-1}}|f(|r|y')|^p \omega_1(y') d\sigma(y') \right) dr \right)^{\frac{1}{p}} dt\\
&\lesssim\|b\|_{Lip^{\beta}}\|\Omega \|_{L^{p'}(S^{n-1})}\int\limits_0^{\infty}\frac{|\Phi(t)|}{t^{1+(\gamma+n)\frac{\lambda_1-1}{p}} (1+t^{-1})^{-\beta}}\sup\limits_{R>0}\left(\frac{1}{\omega_2(B(0,{t^{-1}R}))^{\lambda_1}}\right.\\
&\left.\times\int\limits_{B(0,t^{-1}R)}r^{\gamma+ n-1}\left(\int\limits_{S^{n-1}}|f(|r|y')|^p\omega_1(y') d\sigma(y') \right) dr \right)^{\frac{1}{p}} dt\\
&\lesssim\|b\|_{Lip^{\beta}}.\|\Omega \|_{L^{p'}(S^{n-1})}.\|f\|_{\dot{B}^{p,\lambda_1}(\omega_1,\omega_2)}.\int\limits_0^{\infty}\frac{|\Phi(t)|}{t^{1+(\gamma+n)\frac{\lambda_1-1}{p}} (1+t^{-1})^{-\beta}} dt .
\end{align*} 
 This implies that the commutator  ${{\mathcal H}}_{\Phi,\Omega}^b$ is determined as a bounded operator from $\dot{B}^{p,\lambda_1}(\omega_1,\omega_2)$ to $\dot{B}^{p,\lambda}(\omega_1,\omega_2)$. The proof of the theorem is completed.
\end{proof}
Finally, it is also interesting to give the boundedness of the commutator  ${{\mathcal H}}_{\Phi,\Omega}^b$ on the two weighted Herz type spaces and on the  two weighted Morrey-Herz type spaces. More precisely, we have the results as follows.
\begin{theorem}\label{Herz}
Let $1\le p<\infty,1\le q<\infty$, $0<\beta\le 1$, $\gamma>-n$ and $\alpha_1=\alpha_2+\frac{n\beta}{n+\gamma}$. 
 Suppose $b\in Lip^{\beta} (\mathbb{R}^n)$, $\omega_1, \omega_2\in \mathcal{W}_\gamma$ with $\omega_2(x')\ge c>0$ for all $x'\in S^{n-1}$, $\Omega\in L^{q'}(S^{n-1})$, and
\begin{equation}
\mathcal C_5=\int\limits_0^{\infty}\frac{|\Phi(t)|}{t^{1-{\frac{\gamma}{q}-\frac{n}{q}}-\alpha_1\left(1+\frac{\gamma}{n}\right)}(1+t^{-1})^{-\beta}}    dt< \infty. \notag
\end{equation}
Then the commutator  ${{\mathcal H}}_{\Phi,\Omega}^b$  is a bounded operator from $\dot{K}^{\alpha_1,p}_{q}(\omega_1,\omega_2)$ to $\dot{K}^{\alpha_2,p}_{q}(\omega_1,\omega_2)$.
\end{theorem}
\begin{proof}
Let $f\in \dot{K}^{\alpha_1,p}_{q}(\omega_1,\omega_2)$.  For any $k\in\mathbb{Z}$, by \eqref{ceq1} and the Minkowski inequality, we get
\begin{align*}
&\|{{\mathcal H}}_{\Phi,\Omega}^bf\chi_k \|_{q,\omega_2}\\
&=\left(\int_{C_k} \left|\int\limits_0^{\infty}\int\limits_{S^{n-1}}\frac{\Phi(t)}{t}\Omega(y')f(|x|t^{-1}y') \left(b(x)-b(|x|t^{-1}y') \right)d\sigma(y') dt \right|^{q}\omega_2(x)dx \right)^{\frac{1}{q}}\\
&\le \left(\int_{C_k} \left|\int\limits_0^{\infty}\int\limits_{S^{n-1}}\frac{|\Phi(t)|}{t}\Omega(y')f(|x|t^{-1}y') \left(\|b\|_{Lip^{\beta}}.|x|^{\beta}.(1+t^{-1})^{\beta} \right)d\sigma(y') dt \right|^{q}\omega_2(x)dx \right)^{\frac{1}{q}}\\
&\lesssim \|b\|_{Lip^{\beta}}\left(\int_{C_k} \left|\int\limits_0^{\infty}\int\limits_{S^{n-1}}\frac{|\Phi(t)|}{t(1+t^{-1})^{-\beta}}\Omega(y')f(|x|t^{-1}y') |x|^{\beta} d\sigma(y') dt \right|^{q}\omega_2(x)dx \right)^{\frac{1}{q}}\\
&\lesssim \|b\|_{Lip^{\beta}}\int\limits_0^{\infty}\frac{|\Phi(t)|}{t(1+t^{-1})^{-\beta}}     \left(\int_{C_k} \left|\int\limits_{S^{n-1}}\Omega(y')f(|x|t^{-1}y') |x|^{\beta} d\sigma(y')  \right|^{q}\omega_2(x)dx \right)^{\frac{1}{q}}dt\\
&\lesssim \|b\|_{Lip^{\beta}}|B_k|^{\frac{\beta}{n}}\int\limits_0^{\infty}\frac{|\Phi(t)|}{t(1+t^{-1})^{-\beta}}     \left(\int_{C_k} \left|\int\limits_{S^{n-1}}\Omega(y')f(|x|t^{-1}y') d\sigma(y')  \right|^{q}\omega_2(x)dx \right)^{\frac{1}{q}}dt.
\end{align*}
Using changing variable $u=xt^{-1}$ and by \eqref{eq6a1} again, we obtain
\begin{align}\label{ceq5}
&\|{{\mathcal H}}_{\Phi,\Omega}^bf\chi_k \|_{q,\omega_2}\notag\\
&\lesssim \|b\|_{Lip^{\beta}}|B_k|^{\frac{\beta}{n}}\int\limits_0^{\infty}\frac{|\Phi(t)|}{t^{1-{\frac{\gamma}{q}-\frac{n}{q}}}(1+t^{-1})^{-\beta}}     \left(\int_{\frac{1}{t} C_k} \left|\int\limits_{S^{n-1}}\Omega(y')f(|u|y') d\sigma(y')  \right|^{q}\omega_2(u) du \right)^{\frac{1}{q}}  dt\notag\\
&\lesssim \|b\|_{Lip^{\beta}}\|\Omega \|_{L^{q'}(S^{n-1})}|B_k|^{\frac{\beta}{n}}\int\limits_0^{\infty}\frac{|\Phi(t)|}{t^{1-{\frac{\gamma}{q}-\frac{n}{q}}}(1+t^{-1})^{-\beta}}     \left(\mathcal{J}(t,\omega_2) \right)^{\frac{1}{q}}  dt\notag\\
&\lesssim \|b\|_{Lip^{\beta}}\|\Omega \|_{L^{q'}(S^{n-1})}|B_k|^{\frac{\beta}{n}}\int\limits_0^{\infty}\frac{|\Phi(t)|}{t^{1-{\frac{\gamma}{q}-\frac{n}{q}}}(1+t^{-1})^{-\beta}}     \|f\chi_{\frac{1}{t}C_k} \|_{q,\omega_2}  dt\notag\\
&\lesssim \|b\|_{Lip^{\beta}}\|\Omega \|_{L^{q'}(S^{n-1})}|B_k|^{\frac{\beta}{n}}\int\limits_0^{\infty}\frac{|\Phi(t)|}{t^{1-{\frac{\gamma}{q}-\frac{n}{q}}}(1+t^{-1})^{-\beta}}    \left(\|f\chi_{k+\ell-1} \|_{q,\omega_2}+\|f\chi_{k+\ell} \|_{q,\omega_2} \right)  dt.
\end{align}
where  $\mathcal{J} (t,\omega_2):=\displaystyle\int\limits_{\frac{1}{t}C_k} \left(\displaystyle\int\limits_{S^{n-1}}|f(|u|y')|^{q}d\sigma(y') \right) \omega_2(u) du$, and $\ell=\ell(t)$ is an integer number such that $2^\ell \simeq t^{-1}$.
On the other hand, by the Minkowski inequality for $1\leq p<\infty$, we have
\begin{align*}
&\|{{\mathcal H}}_{\Phi,\Omega}^bf\|_{\dot{K}_{q}^{\alpha_2,p}(\omega_1,\omega_2)}=\left(\sum\limits_{k\in\mathbb{Z}}\omega_1(B_k)^{\alpha_2\frac{p}{n}}\|{{\mathcal H}}_{\Phi,\Omega}^bf\chi_k\|^p_{L^{q}(\mathbb{R}^n;\omega_2)} \right)^{\frac{1}{p}}\\
&\lesssim \|b\|_{Lip^{\beta}}\|\Omega \|_{L^{q'}(S^{n-1})}\left(\sum\limits_{k\in\mathbb{Z}}\omega_1(B_k)^{\alpha_2\frac{p}{n}} \left(|B_k|^{\frac{\beta}{n}}\int\limits_0^{\infty}\frac{|\Phi(t)|}{t^{1-{\frac{\gamma}{q}-\frac{n}{q}}}(1+t^{-1})^{-\beta}}\right.\right.\\
&\left.\left.\Big.\times    \left(\|f\chi_{k+\ell-1} \|_{q,\omega_2}+\|f\chi_{k+\ell} \|_{q,\omega_2} \right)  dt \right)^p \right)^{\frac{1}{p}}\\
&\lesssim \|b\|_{Lip^{\beta}}\|\Omega \|_{L^{q'}(S^{n-1})}\int\limits_0^{\infty}\frac{|\Phi(t)|}{t^{1-{\frac{\gamma}{q}-\frac{n}{q}}}(1+t^{-1})^{-\beta}}\mathcal{B} dt,
\end{align*}
where $\mathcal{B}:=\left(\sum\limits_{k\in\mathbb{Z}}\omega_1(B_k)^{\alpha_2\frac{p}{n}} |B_k|^{\frac{\beta}{n}p}    \left(\|f\chi_{k+\ell-1} \|_{q,\omega_2}+\|f\chi_{k+\ell} \|_{q,\omega_2} \right)^p \right)^{\frac{1}{p}}$.
It is not hard to see that
\begin{align*}
\mathcal{B}\le \left(\sum\limits_{k\in\mathbb{Z}}\omega_1(B_k)^{\alpha_2\frac{p}{n}} |B_k|^{\frac{\beta}{n}p}   \|f\chi_{k+\ell-1} \|_{q,\omega_2}^p \right)^{\frac{1}{p}}+\left(\sum\limits_{k\in\mathbb{Z}}\omega_1(B_k)^{\alpha_2\frac{p}{n}} |B_k|^{\frac{\beta}{n}p}    \|f\chi_{k+\ell} \|_{q,\omega_2}^p \right)^{\frac{1}{p}}.
\end{align*}
Note that it follows from Lemma \ref{lemma} that $\dfrac{|B_k|}{\omega_1(B_k)^{\frac{n}{n+\gamma}}}$ is a constant and
\begin{align}\label{ceq6}
\frac{\omega_1(B_k)}{\omega_1(B_{k+\ell+i})}=2^{-(\ell+i)(n+\gamma)},\;\; i=-1,0.
\end{align}
With $\alpha_1=\alpha_2+\dfrac{n\beta}{n+\gamma}$, by $2^\ell\simeq t^{-1}$, we obtain
\begin{align*}
\mathcal{B}&\le \left(\sum\limits_{k\in\mathbb{Z}}\omega_1(B_{k+\ell-1})^{\alpha_1\frac{p}{n}}    \|f\chi_{k+\ell-1} \|_{q,\omega_2}^p \right)^{\frac{1}{p}}\left(\frac{\omega_1(B_{k})}{\omega_1(B_{k+\ell-1})}\right)^{\frac{\alpha_1}{n}}\left(\frac{|B_k|}{\omega_1(B_k)^{\frac{n}{n+\gamma}}} \right)^{\frac{\beta}{n}}               \\
&+ \left(\sum\limits_{k\in\mathbb{Z}}\omega_1(B_{k+\ell})^{\alpha_1\frac{p}{n}}    \|f\chi_{k+\ell} \|_{q,\omega_2}^p \right)^{\frac{1}{p}}\left(\frac{\omega_1(B_{k})}{\omega_1(B_{k+\ell})}\right)^{\frac{\alpha_1}{n}}\left(\frac{|B_k|}{\omega_1(B_k)^{\frac{n}{n+\gamma}}} \right)^{\frac{\beta}{n}}\\
&\le \left(\left(\frac{\omega_1(B_{k})}{\omega_1(B_{k+\ell-1})}\right)^{\frac{\alpha_1}{n}}  +\left(\frac{\omega_1(B_{k})}{\omega_1(B_{k+\ell})}\right)^{\frac{\alpha_1}{n}}   \right)\left(\frac{|B_k|}{\omega_1(B_k)^{\frac{n}{n+\gamma}}} \right)^{\frac{\beta}{n}}        \|f\|_{\dot{K}_q^{\alpha_1,p}(\omega_1,\omega_2)}\\
&\lesssim \left( 2^{-(\ell-1)\alpha_1\left(1+\frac{\gamma}{n}\right)} +  2^{-\ell \alpha_1\left(1+\frac{\gamma}{n}\right)}\right)  \|f\|_{\dot{K}_q^{\alpha_1,p}(\omega_1,\omega_2)}\\
&\lesssim \left(\frac{1}{t}\right)^{-\alpha_1\left(1+\frac{\gamma}{n}\right)} \|f\|_{\dot{K}_q^{\alpha_1,p}(\omega_1,\omega_2)}.
\end{align*}
Consequently,
\begin{align*}
&\|{{\mathcal H}}_{\Phi,\Omega}^bf\|_{\dot{K}_{q}^{\alpha_2,p}(\omega_1,\omega_2)}\\
&\lesssim \|b\|_{Lip^{\beta}}\|\Omega \|_{L^{q'}(S^{n-1})}\int\limits_0^{\infty}\frac{|\Phi(t)|}{t^{1-{\frac{\gamma}{q}-\frac{n}{q}}}(1+t^{-1})^{-\beta}}  \left( \frac{1}{t}\right)^{-\alpha_1\left(1+\frac{\gamma}{n}\right)}  \|f\|_{\dot{K}_q^{\alpha_1,p}(\omega_1,\omega_2)}  dt\\
&\lesssim \|b\|_{Lip^{\beta}}\|\Omega \|_{L^{q'}(S^{n-1})} \|f\|_{\dot{K}_q^{\alpha_1,p}(\omega_1,\omega_2)} \int\limits_0^{\infty}\frac{|\Phi(t)|}{t^{1-{\frac{\gamma}{q}-\frac{n}{q}}-\alpha_1\left(1+\frac{\gamma}{n}\right)}(1+t^{-1})^{-\beta}}   dt.
\end{align*}
Therefore, the theorem is completely proved.
\end{proof}
Similarly, we also have the following result for the two weighted Morrey-Herz spaces. 
\begin{theorem}
Let $0< p<\infty,1\le q<\infty$, $0<\beta\le 1$, $\gamma>-n$ and $\alpha_1=\alpha_2+\frac{n\beta}{n+\gamma}$. 
 Suppose $b\in Lip^{\beta} (\mathbb{R}^n)$, $\omega_1, \omega_2\in \mathcal{W}_\gamma$ with $\omega_2(x')\ge c>0$ for all $x'\in S^{n-1}$, $\Omega\in L^{q'}(S^{n-1})$, and
\begin{equation}
\mathcal C_5=\int\limits_0^{\infty}\frac{|\Phi(t)|}{t^{1-{\frac{\gamma}{q}-\frac{n}{q}+(\lambda-\alpha_1)\left(1+\frac{\gamma}{n}\right)}}(1+t^{-1})^{-\beta}}dt < \infty. \notag
\end{equation}
Then the commutator  ${{\mathcal H}}_{\Phi,\Omega}^b$  is a bounded operator from $M\dot{K}_{p,q}^{\alpha_1,\lambda}(\omega_1,\omega_2)$ to $M\dot{K}_{p,q}^{\alpha_2,\lambda}(\omega_1,\omega_2)$.
\end{theorem}

\begin{proof}
The proof of the theorem is quite similar to one of Theorem \ref{Herz}, but to convenience to the readers, we also give the brief proof here. In order to estimate the right hand side of \eqref{ceq5}, we need to consider the following  two cases.

Case 1: $1\le p<\infty$. By Minkowski's inequality, it follows from Lemma \ref{lemma} and \eqref{ceq6} that
\begin{align*}
&\|{{\mathcal H}}_{\Phi,\Omega}^bf \|_{M\dot{K}_{p,q}^{\alpha_2,\lambda}(\omega_1,\omega_2)}\\
&\lesssim \|b\|_{Lip^{\beta}}\|\Omega \|_{L^{q'}(S^{n-1})}\sup\limits_{k_0\in\mathbb{Z}}\left(\omega_1(B_{k_0})^{-\frac{\lambda}{n}}\left(\sum\limits_{k=-\infty}^{k_0}\omega_1(B_k)^{\alpha_2\frac{p}{n}}\left( |B_k|^{\frac{\beta}{n}} \right.\right.\right.\\
&\left.\left.\left.\times\int\limits_0^{\infty}\frac{|\Phi(t)|}{t^{1-{\frac{\gamma}{q}-\frac{n}{q}}}(1+t^{-1})^{-\beta}}    \left(\|f\chi_{k+\ell-1} \|_{q,\omega_2}+\|f\chi_{k+\ell} \|_{q,\omega_2} \right)  dt \right)^p \right)^{\frac{1}{p}}\right)\\
&\lesssim \|b\|_{Lip^{\beta}}\|\Omega \|_{L^{q'}(S^{n-1})}\int\limits_0^{\infty}\frac{|\Phi(t)|}{t^{1-{\frac{\gamma}{q}-\frac{n}{q}}}(1+t^{-1})^{-\beta}} \widetilde{\mathcal{B}}  dt,
\end{align*}
where
\[\widetilde{\mathcal{B}}:=\sup\limits_{k_0\in\mathbb{Z}}\left(\omega_1(B_{k_0})^{-\frac{\lambda}{n}}\left(\sum\limits_{k=-\infty}^{k_0}\omega_1(B_k)^{\alpha_2\frac{p}{n}}\left( |B_k|^{\frac{\beta}{n}}    \left(\|f\chi_{k+\ell-1} \|_{q,\omega_2}+\|f\chi_{k+\ell} \|_{q,\omega_2} \right)   \right)^p \right)^{\frac{1}{p}}\right),\] 
and $\ell=\ell(t)$ is an integer number such that $2^\ell \simeq t^{-1}$.
It is clear that
\begin{align*}
\widetilde{\mathcal{B}}&\le \sup\limits_{k_0\in\mathbb{Z}}\omega_1(B_{k_0})^{-\frac{\lambda}{n}}\left(\sum\limits_{k=-\infty}^{k_0}\omega_1(B_{k+\ell-1})^{\alpha_1\frac{p}{n}}    \|f\chi_{k+\ell-1} \|_{q,\omega_2}^p \right)^{\frac{1}{p}}\left(\frac{\omega_1(B_{k})}{\omega_1(B_{k+\ell-1})}\right)^{\frac{\alpha_1}{n}}\left(\frac{|B_k|}{\omega_1(B_k)^{\frac{n}{n+\gamma}}} \right)^{\frac{\beta}{n}}               \\
&+ \sup\limits_{k_0\in\mathbb{Z}}\omega_1(B_{k_0})^{-\frac{\lambda}{n}}\left(\sum\limits_{k=-\infty}^{k_0}\omega_1(B_{k+\ell})^{\alpha_1\frac{p}{n}}    \|f\chi_{k+\ell} \|_{q,\omega_2}^p \right)^{\frac{1}{p}}\left(\frac{\omega_1(B_{k})}{\omega_1(B_{k+\ell})}\right)^{\frac{\alpha_1}{n}}\left(\frac{|B_k|}{\omega_1(B_k)^{\frac{n}{n+\gamma}}} \right)^{\frac{\beta}{n}}\\
&\le \left(\left(\frac{\omega_1(B_{k})}{\omega_1(B_{k+\ell-1})}\right)^{\frac{\alpha_1}{n}}  +\left(\frac{\omega_1(B_{k})}{\omega_1(B_{k+\ell})}\right)^{\frac{\alpha_1}{n}}   \right)\left(\frac{|B_k|}{\omega_1(B_k)^{\frac{n}{n+\gamma}}} \right)^{\frac{\beta}{n}}        \|f\|_{M\dot{K}_{p,q}^{\alpha_1,\lambda}(\omega_1,\omega_2)}\\
&\lesssim 2^{\ell \lambda\left(1+\frac{\gamma}{n}\right)}\left( 2^{-(\ell-1)\alpha_1\left(1+\frac{\gamma}{n}\right)} +  2^{-\ell \alpha_1\left(1+\frac{\gamma}{n}\right)}\right)  \|f\|_{M\dot{K}_{p,q}^{\alpha_1,\lambda}(\omega_1,\omega_2)}\\
&\lesssim \left(\frac{1}{t}\right)^{(\lambda-\alpha_1)\left(1+\frac{\gamma}{n}\right)} \|f\|_{M\dot{K}_{p,q}^{\alpha_1,\lambda}(\omega_1,\omega_2)}.
\end{align*}
Consequently, we have
\begin{align*}
&\|{{\mathcal H}}_{\Phi,\Omega}^bf\|_{M\dot{K}_{p,q}^{\alpha_2,\lambda}(\omega_1,\omega_2)}\\
&\lesssim \|b\|_{Lip^{\beta}}\|\Omega \|_{L^{q'}(S^{n-1})}\int\limits_0^{\infty}\frac{|\Phi(t)|}{t^{1-{\frac{\gamma}{q}-\frac{n}{q}}}(1+t^{-1})^{-\beta}}\left(\frac{1}{t}\right)^{(\lambda-\alpha_1)\left(1+\frac{\gamma}{n}\right)}   \|f\|_{M\dot{K}_{p,q}^{\alpha_1,\lambda}(\omega_1,\omega_2)} dt\\
&\lesssim \|b\|_{Lip^{\beta}}\|\Omega \|_{L^{q'}(S^{n-1})}\int\limits_0^{\infty}\frac{|\Phi(t)|}{t^{1-{\frac{\gamma}{q}-\frac{n}{q}+(\lambda-\alpha_1)\left(1+\frac{\gamma}{n}\right)}}(1+t^{-1})^{-\beta}}  \|f\|_{M\dot{K}_{p,q}^{\alpha_1,\lambda}(\omega_1,\omega_2)} dt.\\
\end{align*}

Case 2: $0<p<1$. We first observe that
\begin{align*}
\|f\chi_{k+\ell+i}\|_{q,\omega_2}&\le \omega_1(B_{k+\ell+i})^{\frac{\lambda-\alpha_1}{n}}\omega_1(B_{k+\ell+i})^{-\frac{\lambda}{n}}\left(\sum\limits_{j=-\infty}^{k+\ell+i}\omega_1(B_j)^{\frac{\alpha_1 p}{n}}\|f\chi_j\|_{q,\omega_2}^p \right)^{\frac{1}{p}}\\
&\le \omega_1(B_{k+\ell+i})^{\frac{\lambda-\alpha_1}{n}} \|f\|_{M\dot{K}_{p,q}^{\alpha_1,\lambda}(\omega_1,\omega_2)}, \;\;i=-1,0.
\end{align*}
Combining this with \eqref{ceq5}, we obtain
\begin{align*}
&\|{{\mathcal H}}_{\Phi,\Omega}^bf \|_{M\dot{K}_{p,q}^{\alpha_2,\lambda}(\omega_1,\omega_2)}\\
&\lesssim \sum_{i=-1,0}\|b\|_{Lip^{\beta}}\|\Omega \|_{L^{q'}(S^{n-1})}\|f\|_{M\dot{K}_{p,q}^{\alpha_1,\lambda}(\omega_1,\omega_2)} \sup\limits_{k_0\in\mathbb{Z}}\left(\omega_1(B_{k_0})^{-\frac{\lambda}{n}}\left(\sum\limits_{k=-\infty}^{k_0}\omega_1(B_k)^{\alpha_2\frac{p}{n}}|B_k|^{\frac{\beta p}{n}}\right.\right.\\
&\left.\left.\times\left( \int\limits_0^{\infty}\frac{|\Phi(t)|}{t^{1-{\frac{\gamma}{q}-\frac{n}{q}}}(1+t^{-1})^{-\beta}}   \omega_1(B_{k+\ell+i})^{\frac{\lambda-\alpha_1}{n}}  dt \right)^p \right)^{\frac{1}{p}}\right)\\
&\lesssim \sum_{i=-1,0}\|b\|_{Lip^{\beta}}\|\Omega \|_{L^{q'}(S^{n-1})}\|f\|_{M\dot{K}_{p,q}^{\alpha_1,\lambda}(\omega_1,\omega_2)} \sup\limits_{k_0\in\mathbb{Z}}\left(\sum\limits_{k=-\infty}^{k_0}\left( \int\limits_0^{\infty}\frac{|\Phi(t)|}{t^{1-{\frac{\gamma}{q}-\frac{n}{q}}}(1+t^{-1})^{-\beta}} \mathcal{T}  dt \right)^p \right)^{\frac{1}{p}},
\end{align*}
where $\mathcal{T}:=\omega_1(B_{k_0})^{-\frac{\lambda}{n}} \omega_1(B_k)^{\frac{\alpha_2}{n}}|B_k|^{\frac{\beta }{n}}  \omega_1(B_{k+\ell+i})^{\frac{\lambda-\alpha_1}{n}}$.
Hence, by \eqref{ceq6} and for any $k\le k_0$, it follows that
\begin{align*}
\mathcal{T}\lesssim 2^{(k-k_0)\left(1+\frac{\gamma}{n}\right)\lambda} \left(\frac{1}{t}\right)^{(\lambda-\alpha_1)\left(1+\frac{\gamma}{n}\right)}.
\end{align*}
Consequently, we obtain
\begin{align*}
&\|{{\mathcal H}}_{\Phi,\Omega}^bf \|_{M\dot{K}_{p,q}^{\alpha_2,\lambda}(\omega_1,\omega_2)}\\
&\lesssim \|f\|_{M\dot{K}_{p,q}^{\alpha_1,\lambda}(\omega_1,\omega_2)} \sup\limits_{k_0\in\mathbb{Z}}\left(\sum\limits_{k=-\infty}^{k_0}2^{(k-k_0)\left(1+\frac{\gamma}{n}\right)\lambda p}  \right)^{\frac{1}{p}}\int\limits_0^{\infty}\frac{|\Phi(t)|}{t^{1-{\frac{\gamma}{q}-\frac{n}{q}+(\lambda-\alpha_1)\left(1+\frac{\gamma}{n}\right)}}(1+t^{-1})^{-\beta}}dt\\
&\lesssim \|f\|_{M\dot{K}_{p,q}^{\alpha_1,\lambda}(\omega_1,\omega_2)} \int\limits_0^{\infty}\frac{|\Phi(t)|}{t^{1-{\frac{\gamma}{q}-\frac{n}{q}+(\lambda-\alpha_1)\left(1+\frac{\gamma}{n}\right)}}(1+t^{-1})^{-\beta}}dt.
\end{align*}
Therefore, the proof of the theorem is completed.
\end{proof}

{\textbf{Acknowledgments}}. The authors were supported by the Vietnam National Foundation for Science and Technology Development (NAFOSTED).

\bibliographystyle{amsplain}

\begin{thebibliography}{79}

\bibitem{ALP2000} J. Alvarez, J.Lakey, M. Guzm\'{a}n-Partida, \textit{Spaces of bounded λ-central mean oscillation, Morrey spaces, and $\lambda$-central Carleson measures}, Collect. Math., 2000, \textbf{51}(11): 147.

\bibitem{Andersen} K. Andersen and E. Sawyer, \textit{Weighted norm inequalities for the Riemann-Liouville and Weyl fractional integral operators}, Trans. Amer. Math. Soc. \textbf{308} (1988), 547-558.
\bibitem{BM} G. Brown and F. M\'{o}ricz, \textit{Multivariate Hausdorff operators on the spaces $L^p(\mathbb R^n)$}, J. Math. Anal. Appl. \textbf{271} (2002), 443-454.

\bibitem{Chuong2016} N. M. Chuong, \textit{Pseudodifferential operators and wavelets over real and p-adic
fields}, Springer, 2018.



\bibitem{HausdoffCDD} N. M. Chuong, D. V. Duong and K. H. Dung, \textit{Multilinear Hausdorff operators on some function spaces with variable exponent}, arXiv:1709.08185 (2017).


\bibitem{CDH2016} N. M. Chuong, D. V. Duong, H. D. Hung,  \textit{Bounds for the weighted Hardy-Ces\`{a}ro operator and its commutator on Morrey-Herz type spaces,} Z. Anal. Anwend. \textbf{35}(2016) 489-504.

\bibitem{CFL2012} J. Chen, D. Fan and J. Li, \textit{Hausdorff operators on function spaces}, Chin. Ann. Math. \textbf{33B}, 537-556(2012).

\bibitem{Christ} M. Christ and L. Grafakos, \textit{Best constants for two non-convolution inequalities}, Proc. Amer. Math. Soc. \textbf{123} (1995), 1687-1693.

\bibitem{CH2014} N. M. Chuong, H. D. Hung, \textit{Weighted $L^p$ and weighted $BMO$-bounds for a new generalized weighted Hardy-Ces\`{a}ro operator}. Integral Transforms Spec. Funct. \textbf{25}(2014), no. 9, 697-710.

\bibitem{CHH2017} N. M. Chuong, N. T. Hong, H. D. Hung, \textit{Multilinear Hardy-Ces\`{a}ro operator and commutator on the product of Morrey-Herz spaces}, Analysis Math., \textbf{43} (4) (2017), 547-565.

\bibitem{Carton-Lebrun} C. Carton-Lebrun and M. Fosset, \textit{Moyennes et quotients de Taylor dans BMO}, Bull. Soc. Roy. Sci. Li\'{e}ge \textbf{53}, No. 2 (1984), 85-87.

\bibitem{FGLY2015} Z. W. Fu, S. L. Gong, S. Z. Lu and W. Yuan, \textit{Weighted multilinear Hardy operators and commutators}, Forum Math. \textbf{27} (2015), 2825-2851.
\bibitem{Ge} C. Georgakis, \textit{The Hausdorff mean of a Fourier-Stieltjes transform}, Proc. Am. Math. Soc., \textbf{116} (1992), 465 - 471.
\bibitem{Hausdorff} F. Hausdorff, \textit{Summation methoden und Momentfolgen}, I, Math. Z. \textbf{9} (1921), 74-109.

\bibitem{HA2017} A. Hussain and M. Ahmed, \textit{Weak and strong estimates for the commutators of Hausdorff operators}, Math. Ineq. Appl. \textbf{20}, 49-56(2017).

\bibitem{Hurwitz} W. A. Hurwitz, L. L. Silverman, \textit{The consistency and equivalence of certain definitions of summabilities}, Trans. Amer.  Math. Soc. \textbf{18} (1917), 1-20.


\bibitem{Miyachi} A. Miyachi, \textit{ Boundedness of the Ces\`{a}ro operator in Hardy space}, J. Fourier Anal. Appl. \textbf{10} (2004), 83-92.

\bibitem{Mo2005} F. M\'{o}ricz, \textit{Multivariate Hausdorff operators on the spaces $H^1(\mathbb{R}^n)$ and $BMO(\mathbb{R}^n)$}, Analysis Math. \textbf{31}, 31-41(2005).

\bibitem{Go1960} R. R. Goldberg, \textit{Convolutions and general transforms on $L^p$}, Duke Math. J. \textbf{27}, 251-259 (1960).
\bibitem{Li2008} E. Liflyand, \textit{Boundedness of multidimensional Hausdorff operators on $H^1(\mathbb{R}^n)$}, Acta. Sci. Math. (Szeged). \textbf{74}, 845-851(2008).

\bibitem{LM2000} E. Liflyand anf F. M\'{o}ricz, \textit{The Hausdorff operator is bounded on the real Hardy space $H^1(\mathbb{R})$}, Proc. Amer. Math. Soc. \textbf{128}, 1391-1396(2000).

\bibitem{LY1995} Lu, S. Z. and Yang, D. C., \textit{The weighted Herz-type Hardy space and its Aplications}. Sci. China Ser. A \textbf{38} (1995) (6), 662-673.

\bibitem{LYH2008} S. Z. Lu, D. C. Yang, G. E. Hu, \textit{Herz type spaces and their applications}, Beijing Sci. Press (2008).
\bibitem{RF2016} J. Ruan, D. Fan, \textit{Hausdorff operators on the power weighted Hardy spaces}, J. Math. Anal. Appl., \textbf{433} (2016) 31-48.

\bibitem{RL2006} K. S. Rim and J. Lee, \textit{Estimates of weighted Hardy-Littlewood averages on the $p$-adic vector spaces}, J. Math. Anal. Appl. \textbf{324}, 1470-1477(2006).
\bibitem{TXZ2011} C. Tang, F. Xue, Y. Zhou, \textit{Commutators of weighted Hardy operators on Herz-type spaces}, Ann. Pol. Math. \textbf{101} (2011), no. 3, 267 –273.
\bibitem{WZ2016} J. L. Wu, W. J. Zhao, \textit{Boundedness for fractional Hardy-type operator on variable-exponent Herz-Morrey spaces}, Kyoto J. Math. Volume \textbf{56}, Number 4 (2016), 831-845.
\bibitem{Xiao2001} J. Xiao, \textit{$L^p$ and $BMO$ bounds of weighted Hardy-Littlewood Averages}, J. Math. Anal. Appl. \textbf{262} (2001), 660-666.

\end{thebibliography}

\end{document}